\documentclass[twoside,11pt]{article}

\usepackage{jmlr2e}
\usepackage{amsmath}
\usepackage{amssymb}
\usepackage{mathtools}

\usepackage[utf8]{inputenc} 
\usepackage[T1]{fontenc}    
\usepackage{hyperref}       
\usepackage{url}            
\usepackage{booktabs}       
\usepackage{amsfonts}       
\usepackage{nicefrac}       
\usepackage{microtype}      
\usepackage{xcolor}         

\ShortHeadings{Spectral Statistics}{Naeem and Pajic}

\title{Spectral Statistics of the Sample Covariance Matrix for High Dimensional Linear Gaussians: \\ Working Draft}

\author{\name Muhammad Abdullah Naeem \email muhammad.abdullah.naeem@duke.edu \\
       \addr Department of Electrical and Computer Engineering\\
       Duke University\\
       Durham, NC 27708, USA
       \AND
       \name Miroslav Pajic \email miroslav.pajic@duke.edu \\
       \addr Department of Electrical and Computer Engineering\\
       Duke University\\
       Durham, NC 27708, USA}

\editor{}
\begin{document}
\maketitle

\begin{abstract}%
Performance of ordinary least squares(OLS) method for the \emph{estimation of high dimensional stable state transition matrix} $A$(i.e., spectral radius $\rho(A)<1$) from a single noisy observed trajectory of the linear time invariant(LTI)\footnote{Linear Gaussian (LG) in Markov chain literature} system  $X_{-}:(x_0,x_1, \ldots,x_{N-1})$ satisfying
\begin{equation}
    \label{eq:full-covar}
    \nonumber
    x_{t+1}=Ax_{t}+w_{t}, \hspace{10pt} \text{ where } w_{t} \thicksim N(0,I_{n}),
\end{equation}  
  heavily rely on negative moments of the sample covariance matrix: $(X_{-}X_{-}^{*})=\sum_{i=0}^{N-1}x_{i}x_{i}^{*}$ and singular values of  $EX_{-}^{*}$, where $E$ is a rectangular Gaussian ensemble $E=[w_0, \ldots, w_{N-1}]$. Negative moments requires sharp estimates on all the eigenvalues $\lambda_{1}\big(X_{-}X_{-}^{*}\big) \geq \ldots \geq \lambda_{n}\big(X_{-}X_{-}^{*}\big) \geq 0$. Leveraging upon recent results on spectral theorem for non-Hermitian operators in \cite{naeem2023spectral}, along with concentration of measure phenomenon and perturbation theory(Gershgorins' and Cauchys' interlacing theorem) we show that only when $A=A^{*}$, typical order of $\lambda_{j}\big(X_{-}X_{-}^{*}\big) \in \big[N-n\sqrt{N}, N+n\sqrt{N}\big]$ for all $j \in [n]$. However, in \emph{high dimensions} when $A$ has only one distinct eigenvalue $\lambda$ with geometric multiplicity of one, then as soon as eigenvalue leaves \emph{complex half unit disc}, largest eigenvalue suffers from curse of dimensionality: $\lambda_{1}\big(X_{-}X_{-}^{*}\big)=\Omega\big( \lfloor\frac{N}{n}\rfloor e^{\alpha_{\lambda}n} \big)$, while smallest eigenvalue $\lambda_{n}\big(X_{-}X_{-}^{*}\big) \in (0, N+\sqrt{N}]$. Consequently, OLS estimator incurs a \emph{phase transition} and becomes \emph{transient: increasing iteration only worsens estimation error}, all of this happening when the dynamics are generated from stable systems.

\end{abstract}

\begin{keywords}
  Sample Covariance Matrix, Spectral Theorem for non-Hermitian linear operators, Concentration of measure, High dimensional geometry.
\end{keywords}
\section{Introduction and Preliminaries}
Recent advances in high dimensional statistics and drastic improvement in computational capabilities of machines have gained much attention from controls community. First step in completion of any control task is availability of accurate models, which unfortunately do not always exist. To cope with this issue, one can simulate a trajectory of dynamical system and fit in a regression model. In this paper we focus on learning from a full state observation, but these principles will hold true for the broader category of learning from Markovian data: input-output data(\cite{oymak2021revisiting}), online Kalman filtering(\cite{tsiamis2022online}) and more.

Performance of least squares method for learning system parameters by a single observed trajectory of unknown dynamical system has been extensively studied, see e.g., 
\cite{nagaraj2020least},\cite{sarkar2019finite},\cite{simchowitz2018learning} and references therein. A fundamental limitation of these existing work stems from the fact that sample complexity is in terms of eigenvalues of various Grammians: which is expected value of the sample covariance matrix. Apart from insufficiency of expectation to characterize typical order of the relevant quantity, Grammian eigenvalues hide dependence on length of the simulated trajectory $N$ and dimension of the underlying dynamical systems $n$. These issues severely limit our understanding of regression on dependent data and turns out that we can non-asymptoticlly learn correct system parameters under stability assumption for low dimensional(essentially scalar) dynamical system and in high dimensions when underlying state-transition matrix is Hermitian(which is essentially multiple independent scalar dynamical systems). Throughout this paper we will work under high dimensional framework.

\paragraph{Formalizing the approach, conclusions and main results:}
The biggest obstruction in our current understanding of  regression on dependent data is using eigenvalues of the naive' variant of sample covariance matrix i.e., Grammian matrix instead of the random eigenvalues of the sample covariance matrix to conclude OLS performance. In fact this task is challenging: take for example $N \times n$ rectangular Gaussian ensemble $E$ with its elements i.i.d standard normals. Let $\sigma_{1}(E)\geq \sigma_{2}(E)\geq \ldots \sigma_{n}(E)$ be the singular values of the Gaussian ensemble, even though all the enteries are independent but singular values are highly correlated,\cite{rudelson2014recent}. So it comes as no surprise why no attempt had been made towards investigating spectral statistics $\big[\sigma_{j}^{2}(X_{-})\big]_{j \in [n]}$, where recall that $\sigma_{j}^{2}(X_{-})=\lambda_{j}\big(X_{-}X_{-*}\big)$. Let $[y_{j}]_{j \in [n]}$ be the rows of the data matrix $X_{-}$, then notice that sample covariance matrix is essentially    
\begin{gather*}
X_{-}X_{-}^{*}:=
\begin{bmatrix}
     \langle y_{1},y_{1}\rangle &  \ldots & \langle y_{1},y_{k+1} \rangle  &\ldots& \langle y_1,y_{n}\rangle  \\ \vdots & \ldots & \vdots & \ldots & \vdots \\ \langle y_{k+1},y_{1}\rangle &  \ldots & \langle y_{k+1},y_{k+1} \rangle  &\ldots& \langle y_{k+1},y_{n}\rangle \\
     \vdots & \ldots & \vdots & \ldots & \vdots \\
     \langle y_{n},y_{1}\rangle & \ldots & \langle y_{n},y_{k+1} \rangle & \ldots &  \langle y_{n},y_{n}\rangle.   
\end{bmatrix}
\end{gather*}
In order to get accurate insights into least squares on Markovian data, we need not just accurate estimates of the edge of the spectrum $\lambda_{1}\big(X_{-}X_{-*}\big), \lambda_{n}\big(X_{-}X_{-*}\big)$ but that of the bulk as well.
Now if we new for the \emph{concentration behavior or $\big[\big \langle y_{j},y_{k}\big\rangle \big]_{j,k \in [n]}$} then we could use tools from perturbation theory e.g., Gershgorins' theorem: \emph{all the eigenvalues of $X_{-}X_{-}^{*}$ lies inside disc centered at $\langle y_{j},y_{j} \rangle$ with radius $\sum_{k \neq j}^{n} |\langle y_{j} , y_{k} \rangle| $ for some $j \in [n]$} and Cauchy's interlacing theorem: \emph{smallest eigenvalue is upper bounded by typical size of the smallest row and bulk eigenvalues can be interlaced by the eigenvalues of a lower dimensional sample covariance matrix} to get sharp estimates of spectral statistics of the sample covariance matrix \emph{explicit in $N$ and $n$}. Diagonal entries of the sample covariance matrix $\big\langle y_{j},y_{j}\big\rangle$ are essentially typical size of each row and off-diagonal entries \emph{$\big\langle y_{j},y_{k}\big\rangle_{j \neq k}$ are correlation/interaction between the rows}. Mere assumption of spectral radius strictly inside unit disc is insufficient to characterize these interactions and row variances. As observed recently in the work of \cite{naeem2023spectral}, when state transition is hermitian rows of the data matrix are essentially independent: spatial correlations are minor compared to the typical size of the row. However when state-transition matrix is not Hermitian and distinct eigenvalues of $A$ have \emph{large discrepancy between their algebraic and geometric multiplicity then spatial correlation are strong}. Analytic display of strong spatial correlations appear in off diagonal terms: that can become large, leading to transience of OLS even under control theoretic stability assumption. Therefore, along with stability we need to characterize dynamical system based on spatial dependencies of its generated trajectory. We provide spectral statistics for the two extreme cases with same spectral radius but one spatially independent(Hermitian case) and other spatially inseparable,S-w-SSCs: only one distinct eigenvalue with algebraic and geometric multiplicity having a difference of $n-1$, which includes single Jordan block of size $n$. Combining the results from spectral theorem of non-Hermitian operators, perturbation theory and concentration of measure phenomenon: we manage to get first quantitative handle on various spectral statistics of the sample covariance matrix explicit in dimension of the state space $n$ and trajectory length $N$ as shown in:

\setlength{\tabcolsep}{12pt}
\renewcommand*{\arraystretch}{1.2}
\begin{tabular}{lcl}
\hline
\textbf{Typical Size}&\textbf{S-w-SSCs}&\textbf{Hermitian}\\
\hline
Largest Eigenvalue: $\lambda_{1}\big( X_{-}X_{-}^{*}\big)$&$\Omega\bigg(\bigg\lfloor \frac{N}{n} \bigg\rfloor e^{\alpha_{\lambda}n}\bigg)$& $\mathcal{O}\bigg(N+n\sqrt{N}\bigg)$\\
Smallest Eigenvalue: $\lambda_{n}\big( X_{-}X_{-}^{*}\big)$&$\mathcal{O}\bigg(N+\sqrt{N}\bigg)$& $\mathcal{O}\bigg(N+\sqrt{N}\bigg)$\\
Eigenvalue Spread &$\mathcal{O}\big(e^{\alpha_{\lambda}n}\big)$&$\mathcal{O}\big(n\sqrt{N}\big)$\\
Temporal interaction: $\langle y_j,y_j\rangle$ &same order as spatial&$N \pm \mathcal{O}\sqrt{N} $\\
Spatial interaction: $\langle y_{j}, y_{k}\rangle$&same order as temporal&$\pm \mathcal{O}\sqrt{N}$\\
Largest singular value: $\sigma_{1}(EX_{-}^{*})$&$\Omega\bigg(\sqrt{n\bigg \lfloor \frac{N}{n} \bigg\rfloor e^{\alpha_{\lambda}n}}\bigg)$&$\mathcal{O}\bigg(\sqrt{nN}+n\bigg)$\\
\hline
\end{tabular}

\paragraph{Notations and preliminaries:}
$\rho(A)$,  $\|A\|_2$, $\|A\|_{F}$, $det(A)$, $tr(A)$ and $\sigma(A)$   represent the spectral radius,  matrix 2-norm, Frobenius norm, determinant, trace and set of eigenvalues(spectrum) of $A$, respectively. 
When, subscript under norm is not specified, automatically matrix 2-norm is assumed.
For a positive definite matrix $A$, largest and smallest eigenvaues are denoted by $\lambda_{max}(A)$ and $\lambda_{min}(A)$, respectively. 
Associated with every rectangular matrix $F \in \mathbb{R}^{N \times n }$ are its' singular values $\sigma_{1}(F) \geq \sigma_{2}(F), \ldots, \sigma_{n}(F) \geq 0$, where without loss of generality we assume that $N>n$. A variational characterization of each  singular values $\sigma_{k}(F)$ follows from Courant-Fischer:
\begin{align}
 \nonumber \sigma_{k}(F)&=\max_{V \subset \mathbb{R}^{n}: dim(V)=k}\hspace{5pt} \min_{ x \in V \cap \mathcal{S}^{n-1} }\big\|F x\big\| \\ & \nonumber =
    \min_{V \subset \mathbb{R}^{n}: dim(V)=n-k+1} \hspace{5pt} \max_{ x \in V \cap \mathcal{S}^{n-1}}\big\|F x\big\|
\end{align}
 A function $g: \mathbb{R}^{n} \rightarrow \mathbb{R}^{p}$ is Lipschitz with constant $L$ if for every $x,y \in \mathbb{R}^{n} $, $\|g(x)-g(y)\| \leq L\|x-y\|$.  Notations like $O,\Theta$ and $\Omega$ will be used to highlight the dependence (\emph{non-asymptotically}) w.r.t number of iterations $N$ and dimensionality of the state space $n$. If a statement is only asymptotically true we will highlight it separately.

Space of probability measure on  $\mathcal{X}$(continuous space) is denoted by  $\mathcal{P(\mathcal{X})}$ and space of its Borel subsets is represented by $\mathbb{B}\big(\mathcal{P(\mathcal{X})}\big)$.
On a metric space $(\mathcal{X},d)$, for $\mu, \nu \in \mathcal{P(\mathcal{X})}$, we define Wasserstein metric of order $p \in [1, \infty)$~as
\begin{equation}
\label{eq:WM}
    W_{p}^{d} (\nu,\mu)= \bigg(\inf_{(X,Y) \in \Gamma(\nu,\mu)} \mathbb{E}~d^{p}(X,Y)\bigg)^{\frac{1}{p}};
\end{equation}
here, $\Gamma(\nu,\mu) \in P(\mathcal{X}^{2})$, and $(X,Y) \in \Gamma(\nu,\mu)$ implies that random variables $(X,Y)$ follow some probability distributions on $P(\mathcal{X}^{2})$ with marginals $\nu$ and $\mu$. Another way of comparing two probability distributions on $\mathcal{X}$ is via relative entropy, which is defined as
\begin{equation}
\label{eq:ent} 
    H(\nu|| \mu)=\left\{ \begin{array}{lr}
    \int \log\bigg(\frac{d\nu}{d\mu}\bigg) d\nu, & \text{if}~ \nu << \mu,
         \\ +\infty, & \text{otherwise}. 
         \end{array}\right.
\end{equation}
\begin{definition}
  [Measure concentration and Typical Size]  Let $(x_{i})_{i=1}^{N}$ be centered and bounded size on average  i.e., variance of $O(1)$, then one would expect $S_{N}:=\sum_{i=1}^{N}x_{i}$ to vary in an interval of $O(N)$. However, \emph{under sufficient independence assumption} between each individual components $x_1,x_2, \ldots,x_{N}$, the sum concentrates in a much narrower interval of size $O(\sqrt{N})$ i.e., $P(|S_{N}|\geq \delta \sqrt{N}) = O(\frac{1}{\delta^{2}})$ "and in words we will say $|S_{N}| $ has a typical size of $O(\sqrt{N})$ ", or denote it by $S_{N}\in [-\sqrt{N},\sqrt{N}]$. This is because each individual variable to collectively vary in a way to produce deviation of $O(N)$ becomes more and more less likely with increase in variables and this remarkable phenomenon is called \emph{concentration of measure} 

In fact, trajectory from \emph{one-dimensional stable ARMA} model, with $x_{0}=0$:
\begin{equation}
\label{eq:1dimarma}
    x_{i+1}=\lambda x_{i}+w_{i}, \hspace{3pt} w_{i} \thicksim N(0,1) 
\end{equation}
for some $|\lambda|<1$, $w_{i}$ and $w_{j}$ are independent for $i \neq j$, also satisfies this phenomenon, precisely:
\begin{equation}
    P \bigg(|S_{N}|\geq \delta \sqrt{\frac{N-\sum_{i=1}^{N}|\lambda|^{2i}}{1-|\lambda|^2}}\bigg) = O\ \bigg(\frac{1}{\delta^{2}}\bigg)
\end{equation}
This phenomenon is not just limited to deviations of sums and can be extended to smooth functions of `sufficiently independent' random variables via Talagrands' Inequality:
\end{definition} 
\begin{definition}
    We say that the probability measure $\mu$ satisfies the $L_{p}$-transportation cost inequality on $(\mathcal{X}, d)$ if there is some constant $C > 0$ such that for any probability measure $\nu$
\begin{equation}
    \mathcal{W}_{p}^{d} (\nu,\mu) \leq \sqrt{2C H(\nu|| \mu)},
\end{equation}    
and we use $\mu \in T_{p}^{d}(C)$ as a shorthand notation.
\end{definition}
Since for $p>1$, $T_{p}^{d}(C)$ implies $T_{1}^{d}(C)$ which is related to concentration of Lipschitz functions:
\begin{remark}
\label{rm: lipiid}
[Theorem 1.1 in \cite{djellout2004transportation}: $T_{1}^{d}(C)$ is equivalent to Lipschitz function concentration  ] Given any Lipschitz function $f$ on random variable $x$ with underlying distribution $\mu \in T_{1}^{d}(C)$, we have: 
\begin{equation}
    \label{eq:lipiid} \mathbb{P} \Bigg[ \bigg| f(x) -<f>_{\mu}\bigg| > \epsilon \Bigg] \leq 2\exp\bigg(-\frac{ \epsilon ^2}{2C \|f\|_{L(d)}^2}\bigg).
\end{equation}
\end{remark}
As we will be interested in deviations and typical behaviors of random processes, like trajectories following some Markovian dynamics e.t.c, it is important to understand how it varies with number of iterations and dimension of the underlying space, a useful result is:
\begin{theorem}
\label{thm:dim_ind_tal}
[Theorem 1.1 in \cite{talagrand1996transportation}:Dimension independent tensorization of Gaussian Measure] 
Standard normal on any finite dimensional conventional metric space $\mathbb{R}^{n}$ satisfies $T_{2}(1)$. Moreover, for an $\ell_{2}$ additive metric 
\begin{equation}
    d_{(N)} ^2 (x^N,y^N):= \sqrt{\sum\nolimits_{i=1}^{N} d^2(x_i,y_i)},
\end{equation}
on product space $\mathbb{R}^{n^{\otimes N}}$, isotropic Gaussian satsfies $T_{1}^{d_{(N)}^2} \big(1\big)$.
\end{theorem}

From now on throughout the paper we will use $\ell_2$ metric and $d=d_{(N)} ^2 (x^N,y^N)$. A remarkable advantage of preceding result combined with Lipschitz function concentration result is that given a process $x=(x_1,x_2, \ldots,x_{N}) \thicksim \mu_{N}= N\bigg(0,\Sigma_{N}\bigg)$, then $\mu_{N} \in T_{1} \bigg(  \big \|\Sigma_{N}^{\frac{1}{2}} \big \|^{2} \bigg)$ and even though process might have temporal dependencies but as long as one can prove $\big\|\Sigma_{N}^{\frac{1}{2}}\big\|^{2}=O(1)$, dimension independent tensorization follows.
\begin{theorem}
[Lemma A.4 in \cite{tao2010random} Negative second moment  \label{thm:neg_2nd_moment_ineq}] Let $1 \leq d \leq p$ and $Y \in \mathbb{R}^{d \times p}$ be a full rank matrix with singular values, $\sigma_{1}(Y) \geq \sigma_{2}(Y) \ldots \geq \sigma_{d}(Y) $. Let $v_{j}$ be the hyperplane generated by all the rows of $Y$-except the $j-th$ : i.e., span of $y_1,y_2, \ldots, y_{j-1}, y_{j+1}, \ldots, y_{d}$  for $1 \leq j \leq d$, $(e_{j})_{j=1}^{d}$ be the canonical basis of $\mathbb{R}^{d}$, then: 
 \begin{equation}
 \label{eq:neg_2nd_moment_ineq}
     \sum_{j=1}^{d} \sigma_{j}^{-2}(Y)=\sum_{ j=1}^{d}\big \langle \big(YY^{*}\big)^{-1}e_{j} ,e_{j} \big \rangle =\sum_{j=1}^{d} d_{j}^{-2}, 
 \end{equation}
where  $d_{j}^{-2}:=d^{-2}(y_j, v_j)$, distance between $y_{j}$ and the point closest to it in the subspace $v_{j}$
\end{theorem}
An important observation about orthogonal projections which can lead to lower bound on the smallest eigenvalue of the sample covariance matrix :
\begin{corollary}
\label{cor:vimpescpnegmomstuck}
Let $x_{j} \in \mathcal{S}^{p-1}$ such that it is orthogonal to subspace $v_{j}$ and $P_{v_{j}^{\perp}}$ be the orthogonal projection onto subspace orthogonal to $v_{j}$ then using properties of projections and cauchy-schwarz:
\begin{equation}
    \big|\langle y_{j},x_{j}\rangle\big|=\big| \langle P_{v_{j}^{\perp}}(y_j),x_{j} \rangle \big| \leq \big \| P_{v_{j}^{\perp}}(y_j) \big\| =d(y_{j},v_{j}).
\end{equation}
\end{corollary}
Non-asymptotic results trivially  follow after we can recover typical size explicitly in number of iterations $N$ and dimension of the underlying state space $n$ (see work on Talagrands' inequality and method of moments in \cite{naeem2023spectral}), so we will explicitly focus on typical size of various spectral statistics that show up in least squares regression on Markovian data.  

\paragraph{Paper structure:}Remaining part of this section introduces the reader to the model of Linear time invariant(LTI) systems: more specifically how not just the magnitude of the eigenvalues of given linear transformation, but their algebraic and geometric multiplicities are necessary to characterize resulting dynamics entirely. In section \ref{sec:OLSintro}, we summarize the basic geometric aspects of least square regression and how its' performance is sensitive to linear dependence between the given basis function(in Markovian setting basis function are essentially rows of the data matrix: so one should not take for granted spatial correlations). Dependence of error statistics on negative moments of the sample covariance matrix(essentially sharp estimates of all the eigenvalues) is also shown. In section \ref{sec:largest_Sing}, we show that in spatially inseparable case the largest eigenvalue of the sample covariance matrix suffers from the curse of dimensionality when spectral radius is outside half unit disc and this is when OLS turns out to be \emph{transient: longer trajectory can lead to worsening of estimation error} . Typical order of the bulk and smallest eigenvalue is studied in section \ref{sec:eigsmp}, where we combine spectral information of state transformation $A$, $\sigma(A)$ with concentration of measure phenomenon and tools from perturbation theory to offer new insights into spread of bulk eigenvalues and an explicity upper bound on the typical size of the smallest eigenvalue. A heuristic proof of the transient behavior of OLS in spatially inseparable case is provided in section \ref{sec:Tal} via tensorization of \emph{Talagrands' inequality} for stable random dynamical system and it is shown that although high dimensional stable ARMA models satisfy iteration independent Talagrands inequality but the constant can have exponential dependece on dimension of the state space. Conclusion and future work is discussed in \ref{sec:conc}.
\paragraph{Model Specification:}
Model under consideration is a following $n-$ dimensional stable LTI-system with isotropic Gaussian noise(interchangeably called $n-$ dimensional stable ARMA model or $n-$ dimensional stable LG).  
\begin{equation}
\label{eq:LGS}
    x_{t+1}= Ax_t+ w_{t}, \hspace{10pt} \rho(A) <1 \hspace{10pt} \text{and i.i.d }~ w_{t} \thicksim \mathcal{N}(0,\mathcal{I}_n).
\end{equation}
It mixes to stationary distribution $\mu_{\infty} \thicksim \mathcal{N}(0, P_{\infty})$, where $P_{\infty}$ is the unique positive definite solution of the following Lyapunov equation:
\begin{equation}
\label{eq:contgram}
    A^{*}P_{\infty}A-P_{\infty}+I_{n}=0.
\end{equation}


Position or magnitude of eigenvalues associated to a linear operator $A$ only provides partial information about its' properties (for the ease of exposition, throughout this paper we will assume that $A$ does not have any non-trivial null space). In this paper, we will study operators and matrices via their actions on associated invariant subspaces and concepts like generalized eigenevectors. Topic in itself can take a semester of work and we refer the reader to \cite{axler1995down, axler1997linear}. However, we will try here to give the reader a quick intuition of this approach, often at the cost of rigor and thoroughness. One of the advantage of taking this approach is: \emph{control on the $k-th$ power of matrix norm, independent of the basis structure .}   Roughly speaking, algebraic multiplicity of eigenvalues follow from chatacteristic polynomial of the matirx. 
\begin{equation}
    \label{eq:detcharpoly} det(zI-A)= \prod_{i=1}^{K} (z-\lambda_{i})^{m_i},
\end{equation}
where $\lambda_{i}$ are distinct with multiplicity $m_{i}$ and  $\sum_{i=1}^{K} m_i=n$. Since algebraic multiplicity of eigenvalue $\lambda_{i}$ is $m_{i}$ , we denote it by $AM(\lambda_{i})=m_{i}$. Similarly with each $\lambda_{i} \in \sigma(A)$, their is an associated set of eigenvectors and dimension of their span corresponds to geometric multiplicity of $\lambda_{i}$ , which we denote by $GM(\lambda_{i})=dim[N(A-\lambda_{i}I)]$. Recall, from linear algebra:
\begin{lemma}
    Let $A \in \mathbb{C}^{n \times n}$. Then $A$ is diagonalizable iff there is a set of $n$ linearly independent vectors, each of which is an eigenvector of $A$. 
\end{lemma}


So, in a situation where $GM(\lambda_{i}) <AM(\lambda_{i}) $, eigenvectors do not span $\mathbb{C}^{n}$ and one resorts with spanning the underlying state space by direct sum decomposition of $A-$ invariant subspaces (which might be spanned by more that one linearly independent vector, comprising of eigenvector and generalized eigenvectors). 
\begin{definition}
    Given a matrix $A \in \mathbb{C}^{n \times n}$ and a subspace $M \subset \mathbb{C}^{n}$, we say that $M$ is an $A-$ invariant subspace if $AM \subset M$.
\end{definition}

Indeed, preceding philosophy is underlying principal of Jordan canonical forms, but its' geometric intricacies can not be ignored when dealing with High Dimensional estimation and control problems.
\begin{proposition}
\label{prop:dirsuminv}
We can decompose the underlying state space as a direct sum decomposition of $A-$ invariant subspaces $[M_{\lambda_{i}}]_{i=1}^{K}$ denoted by:  
\begin{equation}
\label{eq:directsumAinv}
    \mathbb{C}^{n}= M_{\lambda_{1}} \oplus M_{\lambda_{2}} \oplus \ldots \oplus M_{\lambda_{K}},
\end{equation}    
where $M_{\lambda_{i}}=N(A-\lambda_{i}I)^{m_{i}}$ is called the \textbf{Generalized eigenspace} associated with eiegenvalue $\lambda_{i}$(see e.g., theorem 3.11 in \cite{axler1995down}).
\end{proposition}
\paragraph{Direct sum decomposition of the state space via projections onto $A-$ invariant subspace}
Furthermore, one can define orthogonal projection matrices $[P_{\lambda_{i}}]_{i=1} ^{K}$ associated to these invariant subspaces
and the identity matrix can be written as:
\begin{equation}
    \label{eq:addId} I_{n}=P_{\lambda_{1}} \oplus P_{\lambda_2} \oplus \ldots \oplus P_{\lambda_{K}}.
\end{equation}
Consequently, every $x \in \mathbb{R}^{n}$ can be written uniquely as $\oplus_{i=1}^{K} x_{\lambda_i}$

\begin{theorem}
    [Theorem 6 in \cite{naeem2023spectral}]. For a stable matrix $A$ with $K$ distinct eigenvalues, let discrepancy related to eigenvalue $\lambda_{i}$ be $D_{\lambda_{i}}:= AM(\lambda_{i})-GM(\lambda_{i})$, then 
    \begin{equation}
        \label{eq:qnthdl} \|A^{k}\| \leq \max_{1 \leq i \leq K}  k^{D_{\lambda_{i}}} |\lambda_{i}|^{k} \bigg( \frac{1-|\lambda_{i}|}{1-|\lambda_{i}|^{D_{\lambda_{i}+1}}}\bigg)
    \end{equation}
\end{theorem}
A peculiar case for system identification will turn out to of discrepancy of $n-1$ which we call:
\begin{definition} [Stable with Strong Spatial Correlations, S-w-SSCs]
 Consider the following example of a stable state transition matrix with only one distinct eigenvalue of $\rho \in (0,1)$  
\begin{align}
\label{eq:S-w-SSCS}
J_{n}(\rho):=
\begin{bmatrix}
\rho & 1 & 0 &  \cdots &0 &0 \\
0 & \rho  & 1 & \ddots & 0 & 0\\
0 & 0 & \rho  & \ddots & 0 & 0 \\
0 & 0 & 0 & \ddots & 1 & 0
\\
\vdots & \ddots & \ddots & \ddots & \rho & 1 \\
0 & 0 & 0 & \cdots & 0 & \rho   
\end{bmatrix}
\end{align}  
with algebraic multiplicity of $n$ but only one linearly independent eigenvector.
\end{definition}
\section{Least Squares on Markovian Data}
\label{sec:OLSintro}
In this section we analyse the problem of OLS estimation for system transition matrix $A$ from single observed (as in \cite{sarkar2019near}, \cite{simchowitz2018learning}, \cite{tsiamis2021linear}) trajectory of $(x_0,x_1, \ldots,x_{N})$ satisyfing:
\begin{equation}
    \label{eq:LGS}
    x_{t+1}=Ax_{t}+w_{t}, \hspace{10pt} \text{ where } w_{t} \thicksim N(0,I) .
\end{equation}  
Before delving into solution of the estimation problem, we would like to give a brief overview into working of general OLS regression along with potential limitations:
A priori you are given $k-$ basis functions of an $n-$ dimensional vector space, where $n>k$ and one can form an $n \times k$ matrix $X:=[v_1,v_2,\ldots,v_{k}] \in \mathbb{R}^{n \times k}$. Now one observes $y$ :
\begin{equation}
    y=X\beta^{*}+\epsilon, \hspace{5pt} \epsilon \thicksim N(0,I_{n})
\end{equation}
and tries to estimate $\beta:=\beta(X,y) \in \mathbb{R}^{k \times 1}$ by projecting observation $y$ onto the span of $X$, we get $\beta=(X^{*}X)^{-1}X^{*}y$ and the expected error in $\ell_{2}$ norm is:
\begin{align}
   \mathbb{E}\| \beta-\beta^{*} \|^{2}= Tr([X^{*}X]^{-1})=\sum_{j=1}^{k} d^{-2}_{j}= \sum_{j=1}^{k} \sigma^{-2}_{j}(X),
\end{align}
where last equality follows from negative second moment identity from Theorem \ref{thm:neg_2nd_moment_ineq} and one can immediately conclude: if any column of $X$ gets closer in terms of $\ell_{2}$ distance on $\mathbb{R}^{n}$ to the span of remaining $k-1$ columns, expected squared error in OLS estimation will increase. Vaguely speaking, linear dependence between basis of the data matrix deteriorates the performance of standard OLS regression. 
 
OLS solution for identification of LTI system from a single observed trajectiry is:
\begin{equation}
    \label{eq:OLSsol} \hat{A}= \arg \min_{B \in \mathbb{R}^{n \times n}}  \sum_{t=0}^{N-1} \|x_{t+1}-Bx_{t}\|^{2}.
\end{equation}
Recall, $X_{+}=[x_1, x_2,  \ldots, x_N]$ and $ X_{-}=[x_0, x_1, \ldots, x_{(N-1)}]$, noise covariates $E=[w_0, w_1, \ldots, w_{N-1}]$, $y_{j}$ be the rows of $X_{-}$ and $v_{j}$ be the hyperplane as defined in theorem \ref{thm:neg_2nd_moment_ineq}. Also notice that conditioned on $x_{0}=0$ state at time $i$ can be represented in terms of powers of $A$ and noise covariates as:  
\begin{equation}
\label{eq:dynamgauss}
x_{i}= \sum_{t=1}^{i}A^{i-t}w_{t-1}, \hspace{5pt} \textit{and graphical representation}    
\end{equation}
\[  A^{*}:=
\begin{bmatrix}     
    \vert & \vert & \vert \\
    b_{1} &  \vert  & b_n   \\
    \vert & \vert & \vert 
\end{bmatrix}
\hspace{5pt}, X^{*}_{-}=
\begin{bmatrix}
     \vert & \vert & \vert &  \vert &\vert\\
     \vert & \vert & \vert & \vert &\vert \\
    \vert & \vert & \vert & \vert &\vert \\
    y_{1} & y_{2} & \vert & y_{n-1} &y_{n}   \\
    \vert & \vert & \vert & \vert &\vert \\
    \vert & \vert & \vert & \vert &\vert \\
    \vert & \vert & \vert &  \vert &\vert
\end{bmatrix}
\hspace{5pt},  X^{*}_{+}=
\begin{bmatrix}
    \vert & \vert & \vert &  \vert &\vert\\
     \vert & \vert & \vert & \vert &\vert \\
    \vert & \vert & \vert & \vert &\vert \\
    z_{1} & z_{2} & \vert & z_{n-1} &z_{n}   \\
    \vert & \vert & \vert & \vert &\vert \\
    \vert & \vert & \vert & \vert &\vert \\
    \vert & \vert & \vert &  \vert &\vert
\end{bmatrix}
\]
In this dynamical version of OLS, at most $n$ basis are provided by columns of $X^{*}_{-}$. Then $i-$ th column of $A^{*}$ corresponds to co-effiecients estimated, by orthogonally projecting observation $z_{i}$ onto span of $X^{*}_{-}$
\begin{equation}
    \hat{b}_{i}=\big( X_{-} X^{*}_{-} \big)^{-1} X_{-}z_{i}
\end{equation}
Then the closed form expression for  Least squares solution and error are:
\begin{align}
 & \label{eq:OLS}    \hat{A}= X_{+}X_{-} ^{\dagger}, \hspace{5pt} \textit{where} \hspace{3pt} X_{-} ^{\dagger}:= X_{-} ^{*}(X_{-}X_{-}^{*})^{-1}  \\ & \hspace{30pt} \label{eq:OLSerror} \big\|A-\hat{A}\big\|_{F} =\big\|EX_{-}^{\dagger}\big\|_{F}
\end{align} 
\textbf{Estimation error is independent of basis function choice:}
Recall that if $A=A^{*}$, then there exists a unitary matrix $U \in \mathbb{C}^{n \times n}$ and a real diagonal matrix $\Omega \in \mathbb{R}^{n \times n}$ such that $A=U^{*} \Omega U$ and a coordinate transform for \eqref{eq:LGS}, i.e., $\hspace{3pt} z_{t+1}=\Omega z_t+Uw_t $, where $z_{t}:=Ux_{t}$ and $Uw_t$ is again an isotropic Gaussian.  Now let $W:=UE$, and notice that rows of $Z_{-}:=U X_{-}$ are independent one dimensional Linear Gaussian dynamics with growth/ decay parameter defined by corresponding diagonal entery in $\Omega$. As unitary by definition $U^* U=UU^*=I$, implies $U^{-1}=U^{*}$, we have 
\begin{align}
     & \nonumber \big\| WZ_{-}^{*} \big(Z_{-}Z_{-}^{*}\big)^{-1} \big\|_{F}^{2}=\big\| UE(UX_{-})^{*}(UX_{-}X_{-}^{*}U^{*})^{-1} \big\|_{F}^{2} \\ & \nonumber = Tr(EX_{-}^{*}(X_{-}X_{-}^{*})^{-2}X_{-}E^{*})=\|EX_{-}^{*}(X_{-}X_{-}^{*})^{-1}\|_{F}^{2}.  
\end{align}
i.e., for the Hermitian case all the rows of $X_{-}$ are independent of each other. For non-Hermitian case we can get independent row-blocks, as in 
\newline
\textbf{Statistical independence between the row blocks of the data matrix }
\label{subsec:Sp-gl}
Notice that using the direct sum decomposition of the state space, $i-$ the column of the data matrix can be decomposed into row blocks that are statistically independent of each other:  
\begin{align}
 \label{eq:sdmx} x_{i}= \sum_{t=1}^{i}A^{i-t}w_{t-1}&=\sum_{t=1}^{i} A^{i-t}\bigg( \sum_{m=1}^{K} P_{\lambda_{m}}\bigg) w_{t-1}  = \sum_{m=1}^{K}  \underbrace{\sum_{t=1}^{i}A_{\lambda_{m}}^{i-t} w_{t-1}^{m}}_{:=B_{\lambda_{m}}(i)}.
\end{align}
Now recall that $D_{\lambda_{m}}$ was used to denote discrepancy between algebraic and geometric multiplicity of eigenvalue $\lambda_{m}$ where for fixed $t$ notice that $\mathbb{E}\|w_{t-1}^{m}\|^{2}=\mathbb{E}\big \langle w_{t-1},P_{\lambda_{m}}w_{t-1} \big\rangle=Tr(P_{\lambda_m})=D_{\lambda_{m}}+1$. So $w_{t-1}^{m}$ is standard normal on $A-$ invariant subspace $M_{\lambda_{m}}$ of dimension $D_{\lambda_{m}}+1$. Since for all $t \in \mathbb{N}$ and $m,m^{'} \in [K]$ such that $m \neq m^{'}$, $w_{t}^{m}$ and $w_{t}^{m^{'}}$ are independent: $\mathbb{E}\big[w_{t}^{m} (w_{t}^{m^{'}})^{*}\big]=\mathbb{E}\big[P_{\lambda_{m}}w_{t}(P_{\lambda_{m^{'}}}w_{t})^{*}\big]=\mathbb{E}\big[P_{\lambda_{m}}w_{t}w_{t}^{*}P_{\lambda_{m^{'}}}\big]= P_{\lambda_{m}}P_{\lambda_{m^{'}}}= \delta_{m}(m^{'})I_{n}$, combined with the fact that $A_{m}^{*} A_{m'}=0$ for $m \neq m'$ (because $N(A_{\lambda_{m}}^{*})=\oplus_{j \neq m}^{K} M_{\lambda_{j}}$) implies that $B_{\lambda_{m}}(i)$ and $B_{\lambda_{m'}}(i)$ are independent of each other for all $i$. Therefore, data matrix essentially contains time realization of $K$-low dimensional dynamical systems, which are statistically independent of each other. As we showed in the previous section that estimation error is independent of the choice of underlying basis, w.l.o.g we can take canonical basis implying rows in data matrix comprises of independent blocks, where each block is a time realization of trajectory generated via canonical form of linear transformation with specified eigenvalue and size of the block equals $D_{\lambda_{m}}+1$.
\newline
\textbf{Error Statistics via negative moments of Sample Covariance matrix:} Let $\sigma_{1}(X_{-}) \geq \sigma_{2}(X_{-}) \geq \ldots \sigma_{n}(X_{-})>0$, be the singular values of the data matrix. One can consider following error approximations for least squares:
\begin{enumerate}
    \item 
    \begin{align}
    \label{eq:erroapxlsqforth} \sigma_{n}(EX_{-}^{*}) \sqrt{\sum_{j=1}^{n} \frac{1}{\sigma_{j}^{4}(X_{-})}} \leq \big\|A-\hat{A} \big\|_{F} \leq \sigma_{1}(EX_{-}^{*}) \sqrt{\sum_{j=1}^{n} \frac{1}{\sigma_{j}^{4}(X_{-})}},
\end{align}
    \item 
    \begin{align}
    \label{eq:errorap2mom}
     \sqrt{\sum_{j=1}^{n} \frac{1}{n\sigma_{j}^{2}(X_{-})}} \leq \big\|A-\hat{A} \big\|_{F} \leq  \sqrt{\sum_{j=1}^{n} \frac{n}{\sigma_{j}^{2}(X_{-})}}
    \end{align}
\end{enumerate}
Most important takeaway is that we need to compute negative moments of the sample covariance matrix to really understand performance of least squares. Which is a very difficult task, but if we can somehow get really tight estimate of all the singular values then there is hope. We will first focus on the extreme singular values of $X_{-}$ and then come to bulk and in the end we will analyse extreme singular values of $EX_{-}^{*}$.
\section{Statistics of the Largest Eigenvalue}
\label{sec:spst}
Extreme singular values(largest and smallest singular values) of the data matrix can act as a sanity check for the performance of OLS. They also play a crucial role in numerical linear algebra(see e.g., \cite{vu2007condition} for more details). For rectangular matrices with i.i.d entries(centered and unit variance) everywhere as in Gaussian ensemble $E$, statistics of $\sigma_{1}(E)$ and $\sigma_{n}(E)$ are well known and tend to concentrate around $\sqrt{N}+\sqrt{n}$ and $\sqrt{N}-\sqrt{n}$ respeectively. Understanding upper tail behavior of $\sigma_{1}(E)$ requires discretization of $S^{n-1}$. 
\begin{definition}
 (epsilon net of $S^{n-1}$) Finite subset, $\mathcal{N}_{\epsilon}(S^{n-1})$ called epsilon net of $n$ dimensional sphere- with the property that given any $a \in S^{n-1}$ there exists  $\hat{a} \in \mathcal{N}_{\epsilon}(S^{n-1})$ such that $\|a - \hat{a}\|_{2} \leq \epsilon$   
\end{definition}

\begin{lemma}
    Metric entropy of sphere is exponential w.r.t dimension of the underlying state space, precisely for any $\epsilon \in (0,1] $:
    \begin{equation}
        |\mathcal{N}_{\epsilon}(S^{n-1})| \leq \bigg( \frac{3}{\epsilon}\bigg)^{n}.
    \end{equation}
\end{lemma}
\label{sec:largest_Sing}
 Recall, $\sigma_{1}(E)=\sup_{a \in S^{n-1}} \|E^{*}a\|$ but $E$ is random and one does not know in advance for which $a^{'} \in S^{n-1}$ supremum is achieved, neither we can upper bound $P( \sup_{a \in S^{n-1}} \|E^{*}a\| > \delta)$ via union bound as the $S^{n-1}$ is not countable. To circumvent this issue, one can discretize $S^{n-1}$ into finite subset, $\mathcal{N}_{\epsilon}(S^{n-1})$ when combined with triangle inequality, reveals: 
 \begin{equation}
     \sigma_{1}(E) \leq \bigg(\frac{1}{1-\epsilon}\bigg) \sup_{a \in \mathcal{N}_{\epsilon}(S^{n-1})} \|E^{*}a\|.
 \end{equation}
 Therefore,
\begin{align}
      \nonumber & \mathbb{P}\bigg( \sup_{a \in S^{n-1}} \|E^{*}a\| \geq \delta \bigg) \leq \mathbb{P}\bigg( \sup_{a \in \mathcal{N}_{\epsilon}(S^{n-1})} \|E^{*}a\| \geq (1-\epsilon) \delta \bigg) \\ & \label{eq:typpoint} \leq \sum_{a \in \mathcal{N}_{\epsilon}(S^{n-1})} \mathbb{P}\bigg(\|E^{*}a\| \geq (1-\epsilon)\delta \bigg) \leq \bigg(\frac{3}{\epsilon}\bigg)^{n} \mathbb{P}\bigg( \|z_{N}\|^{2} \geq (1-\epsilon)^{2} \delta^{2} \bigg),  
\end{align}
where, first inequality in \eqref{eq:typpoint} follows from union bound and second from metric entropy of $n$ dimensional unit sphere. Notice that: \emph{regardless of exact realization of `$a$' on $n$ dimensional unit sphere}, $E^{*}a=z_{N} \thicksim N(0,I_{N})$ and a simple exponential moment calculation on the last expression of \eqref{eq:typpoint} will reveal the correct behavior $\sigma_{1}(E)$. Similar discretization approaches have been considered in literature for understanding statistics of extreme singular values of the data matrix. However, as suggested by authors work in \cite{naeem2023spectral} on spectral theorem for non-Hermitian linear operators, discrepancy between algebraic and geometric multiplicities of distinct eigenvalues of $A$ add structure to data matrix and naive discretization will be wasteful; except for Hermitian and all the same eigenvalues. Furthermore, rows of the data matrix  $X_{-}$ can be divided into blocks, independent of each other and concentration of $\|X_{-}^{*}a\|_{2}$ for some $a \in S^{n-1}$ can be better understood by also restricting $a$ onto invariant subspaces $[M_{\lambda_{i}}]_{i \in [K]}$  of $A$, which we now explore in further detail. 
  
\subsection{Lower bound via typical size of rows}
Notice that,
\begin{align}
    \label{eq:epnet} &\mathbb{P}\bigg( \|X_{-}^{*}a\| \geq \delta  \bigg) =  \mathbb{P}\Bigg( \big\|\sum_{m=1}^{K}X_{-}^{*}P_{\lambda_{m}}a\big\|^2 \geq \delta^{2}  \Bigg), 
\end{align}
where $[X_{-}^{*}P_{\lambda_{m}}a]_{m=1}^{K}$ are independent of each other: orthogonal projections of Gaussians are independent. Furthermore, OLS error in Frobenius norm is independent of choice of the basis, so w.l.o.g we can assume $X_{-}^{*}P_{\lambda_{m}}a \in \mathbb{R}^{N^{\otimes D_{\lambda_{m}}^{+}}}$, where we use shorthand $D_{\lambda_{m}}^{+}:=(D_{\lambda_{m}}+1)$ and recall $\sum_{m=1}^{K}D_{\lambda_{m}}^{+}=n$. , for each $m \in [K]$ let $S_{m}(\lambda):= \sum_{i=1}^{m-1}D_{\lambda_{i}}^{+}$
\begin{equation}
P_{\lambda_{m}}a:=\big(0,\ldots,0,a_{S_{m}(\lambda)+1},\ldots,a_{S_{m}(\lambda)+D_{\lambda_{m}}^{+}},0,\ldots,0\big)    
\end{equation}
Recall that $\sum_{i=1}^{N} [X_{-}]_{j,i} =  \sum_{i=1}^{N} \bigg[ \sum_{t=1}^{i} \langle A^{i-t}w_{t-1},e_{j}\rangle  \bigg]$, therefore:  
\begin{align}
     X_{-}^{*}P_{\lambda_{m}}a= \bigg[a_{S_{m}(\lambda)+1} \overline{[X_{-}]}_{[S_{m}(\lambda)+1,:]}, a_{S_{m}(\lambda)+2} \overline{[X_{-}]}_{[S_{m}(\lambda)+2,:]}, \ldots, a_{S_{m}(\lambda)+D_{\lambda_{m}}^{+}} \overline{[X_{-}]}_{[S_{m}(\lambda)+D_{\lambda_{m}}^{+},:]} \bigg],
\end{align}
where we used $:$ to denote all the columns of the data matrix. For each time $t \in \mathbb{N}$, $w_{t}^{m} \in N \big(0,I_{D_{\lambda_{m}}^{+}}\big)$ i.e., i.i.d isotropic Gaussian of dimension $D_{\lambda_{m}}^{+}$. Hence, for $i \in [N]$ and $l \in [D_{\lambda_{m}}^{+}]$:
\begin{align}
    & \nonumber \overline{[X_{-}]}_{[S_{m}(\lambda)+l,i]}= \sum_{t=1}^{i} \overline{\langle A_{\lambda_{m}}^{i-t}w_{t-1}^{m},e_{l}\rangle}=\sum_{t=1}^{i} \langle e_{l} ,A_{\lambda_{m}}^{i-t}w_{t-1}^{m}\rangle= \sum_{t=1}^{i} \langle e_{l} ,(\lambda_{m}I+N)^{i-t}w_{t-1}^{m}\rangle \\ & \label{eq:indblockdynwillleadtolargsing} =\sum_{t=1}^{i} \sum_{p=0}^{i-t} \binom{i-t}{p} \overline{\lambda_{m}^{i-t-p}}  \langle e_{l+p} , w_{t-1}^{m}\rangle= \sum_{t=1}^{i} \sum_{p=0}^{(i-t) \land (D_{\lambda_{m}}^{+}-l)} \binom{i-t}{p} \overline{\lambda_{m}^{i-t-p}}  \langle e_{l+p} , w_{t-1}^{m}\rangle
\end{align}
where $A_{\lambda_{m}}^{i-t}$ is a Linear transformation from and onto an $D_{\lambda_{m}}^{+}$ dimensional $A$- invariant subspace, i.e., Jordan block of size $D_{\lambda_{m}}^{+}$ with eigenvalue $\lambda_{m}$  Consequently,
\begin{align}
     \label{eq:sig1iidblock} \|X_{-}^{*}P_{\lambda_{m}}a\|^2 & = \sum_{i=1}^{N} \Bigg(\sum_{l=1}^{D_{\lambda_{m}}^{+}} \sum_{t=1}^{i} a_{S_{m}(\lambda)+l} \langle e_{l} ,A_{\lambda_{m}}^{i-t}w_{t-1}^{m}\rangle \Bigg)^{2} \\ & = \sum_{i=1}^{N} \Bigg(\sum_{l=1}^{D_{\lambda_{m}}^{+}} a_{S_{m}(\lambda)+l} \sum_{t=1}^{i}  \sum_{p=0}^{(i-t) \land (D_{\lambda_{m}}^{+}-l)} \binom{i-t}{p} \overline{\lambda_{m}^{i-t-p}}  \langle e_{l+p} , w_{t-1}^{m}\rangle \Bigg)^{2}.
\end{align}
Therefore,
\begin{align}
     & \nonumber \mathbb{P}\bigg( \big\|X_{-}^{*}a\big\| \geq \delta \bigg)= \mathbb{P}\bigg(  \sum_{m=1}^{K} \big\|X_{-}^{*}P_{\lambda_{m}}a \big\|^2 \geq \delta^{2}  \bigg)\\ & \label{eq:spdecompsig1}  \leq 
     e^{-s\delta^{2}}  \prod_{m=1}^{K} \mathbb{E}\bigg[e^{s\sum_{i=1}^{N}\big(\sum_{l=1}^{D_{\lambda_{m}}^{+}} a_{S_{m}(\lambda)+l} \sum_{t=1}^{i}  \sum_{p=0}^{(i-t) \land (D_{\lambda_{m}}^{+}-l)} \binom{i-t}{p} \overline{\lambda_{m}^{i-t-p}}  \langle e_{l+p} , w_{t-1}^{m}\rangle \big)^{2}}\bigg],  
\end{align}
where inequality \eqref{eq:spdecompsig1} follows from Markov inequality combined with independence of $[X_{-}^{*}P_{\lambda_{m}}a]$ for each $m \in [K]$(spatial independence between row blocks of the data matrix). Now notice that if $[\lambda_{m}]_{m \in [K]}$ are all positive, then  for each $m \in [K]$, and $i \in [N]$ sufficiently large with $l$ as variable,
\begin{equation}
    \label{eq:largvarblock}
    a_{S_{m}(\lambda)+l}\sum_{t=1}^{i}  \sum_{p=0}^{(i-t) \land (D_{\lambda_{m}}^{+}-l)} \binom{i-t}{p} \overline{\lambda_{m}^{i-t-p}}  \langle e_{l+p} , w_{t-1}^{m}\rangle,
\end{equation}
will have the largest typical size for $l=1$, so supremum over all $a \in S^{D_{\lambda_{m}}}$ for
\newline
\vspace{2pt}
$\big(\sum_{l=1}^{D_{\lambda_{m}}^{+}} a_{S_{m}(\lambda)+l} \sum_{t=1}^{i}  \sum_{p=0}^{(i-t) \land (D_{\lambda_{m}}^{+}-l)} \binom{i-t}{p} \overline{\lambda_{m}^{i-t-p}}  \langle e_{l+p} , w_{t-1}^{m}\rangle \big)^{2}$, would be in fact  
\begin{equation}
    \label{eq:sparsenettoprow}
  \bigg(\sum_{t=1}^{i}  \sum_{p=0}^{(i-t) \land (D_{\lambda_{m}}^{+}-l)} \binom{i-t}{p} \overline{\lambda_{m}^{i-t-p}}  \langle e_{l+p} , w_{t-1}^{m}\rangle\bigg)^{2}. 
\end{equation}
So when discrepancies of eigenvalues are large typical size of the row should be a good approximate of $\sigma_{1}(X_{-})$.
Thus, at least when rows are correlated we do not need to discretize the unit sphere. 
Also notice that 
\begin{align}
     \label{eq:expro}P \bigg( \sigma_{1}(X_{-}) \leq \delta \bigg)= P \bigg( \sup_{a \in S^{n-1}} \|X_{-}^{*}a\| \leq \delta \bigg)= P \bigg( \sup_{a \in S^{n-1}} \big\|\sum_{j=1}^{n}a_{j}y_{j}\big\| \leq \delta \bigg)  \leq P\bigg( \max_{j \in [n]} \|y_{j} \|_{2} \leq \delta \bigg), 
\end{align}
where $[y_{j}]_{j=1}^{n}$ are the rows of the data matrix as in Section \ref{sec:OLSintro} and if $k=\arg \max_{j \in [n]}\|y_{j}\|$ then last inequality follows by setting $a_{k}=1$ and $a_{j}=0$ for $j \neq k$.
\begin{remark}
If the typical size of any row of the data matrix is `large' then so is the magnitude of $\sigma_{1}(X_{-})$. Typical size of the row is essentially dictated by non-Hermitian variant of the spectral theorem which define the weighted constants on standard normals that populate any given row of the data matrix. 
\subsection{Covariates from stable dynamics can suffer from the curse of dimensionality}
\label{sec:spatio-temp}
It was noted in \cite{naeem2023spectral}, for the case of dynamics correspodning to S-w-SSCs, each covariate in $j-$ th row of the data matrix is a weighted sum of i.i.d standard normals that make up rows $[j, \ldots, n]$ of Gaussian Ensemble $E$. So as one would suspect that first row is most susceptible to having the largest typical size. Which turns out to be exponential in dimension of the state space    
\end{remark}

\begin{proposition}
    \label{lm:cov-ul-bd}
    $N > n$ and $\lambda \in (\frac{1}{2},1)$ then typical order of first row of the data matrix corresponding to dynamics generated from $n$ dimensional S-w-SSCs is exponential in $n$, precisely said
\begin{flalign}
 & \label{eq:covariatecursed} \|y_1\| \geq \Omega\bigg(\bigg\lfloor \frac{N}{n} \bigg\rfloor^{\frac{1}{2}}  e^{\frac{\alpha_{\lambda}n}{2}} \bigg) 
\end{flalign}
where $\alpha_{\lambda}$ entirely depends on $\lambda$.
\end{proposition}
\begin{proof}
 Let $[\hat{e}_{i}]_{i \in N}$ be the canonical basis of $\mathbb{R}^{N}$. As, $j-$ th row and $i-$ th column of the data matrix corresponding to S-w-SSCs case can be compactly written as:
\begin{equation}
\label{eq:inpdatswsscs}
 \langle y_{j}, \hat{e}_{i} \rangle= \sum_{t=1}^{i} \sum_{m=0}^{(i-t) \land (n-j) } \binom{i-t}{m} \lambda^{i-t-m} \big\langle w_{t-1} ,e_{m+j} \big \rangle.   
\end{equation}
Thus revealing: $\langle y_{1},\hat{e}_{n}\rangle$ is a weighted sum of independent $\frac{n(n+1)}{2}$ standard normals, more precisely:
\begin{align}
       \nonumber \langle y_{1},\hat{e}_{n}\rangle&=\sum_{t=1}^{n} \sum_{m=0}^{(n-t)}\binom{n-t}{m}\lambda^{(n-t-m)}\langle w_{t-1} , e_{m+1}\rangle
\end{align}
which is centered with variance: 
\begin{equation}
    \mathbb{E}\big(\langle y_{1},\hat{e}_{n}\rangle^2\big)=\sum_{k=1}^{n}\sum_{m=0}^{n-k} \binom{n-k}{m}^{2} \lambda^{2(n-k-m)}
\end{equation}
upper and lower bounded via Stirling type approximation:
\begin{equation}
   4^{n}\lambda^{2n} \sum_{k=1}^{n}\frac{1}{4^{k}\lambda^{2k} \sqrt{\pi(n-l+\frac{1}{3})}} \leq \mathbb{E}\big(\langle y_{1},\hat{e}_{n}\rangle^2\big) \leq 4^{n}\lambda^{2n}\sum_{k=1}^{n} \frac{1}{4^{k}\lambda^{2k}\sqrt{\pi(n-l+\frac{1}{4})}}
\end{equation}
Let $L_{\lambda,n}(1):=\sum_{k=1}^{n}\frac{1}{4^{k}\lambda^{2k} \sqrt{\pi(n-k+\frac{1}{3})}}$ and $U_{\lambda,n}(1):=\sum_{k=1}^{n}\frac{1}{4^{k}\lambda^{2k} \sqrt{\pi(n-k+\frac{1}{4})}}$
So as soon as $\lambda>\frac{1}{2}$, $n-$ th element of first row has a typical size which is exponential in dimension of the state space. But this trend is periodic: with period $n$, so lower bound follows. 
\end{proof}
\begin{figure} [!t]
\begin{center}
\includegraphics[width=0.75\textwidth]{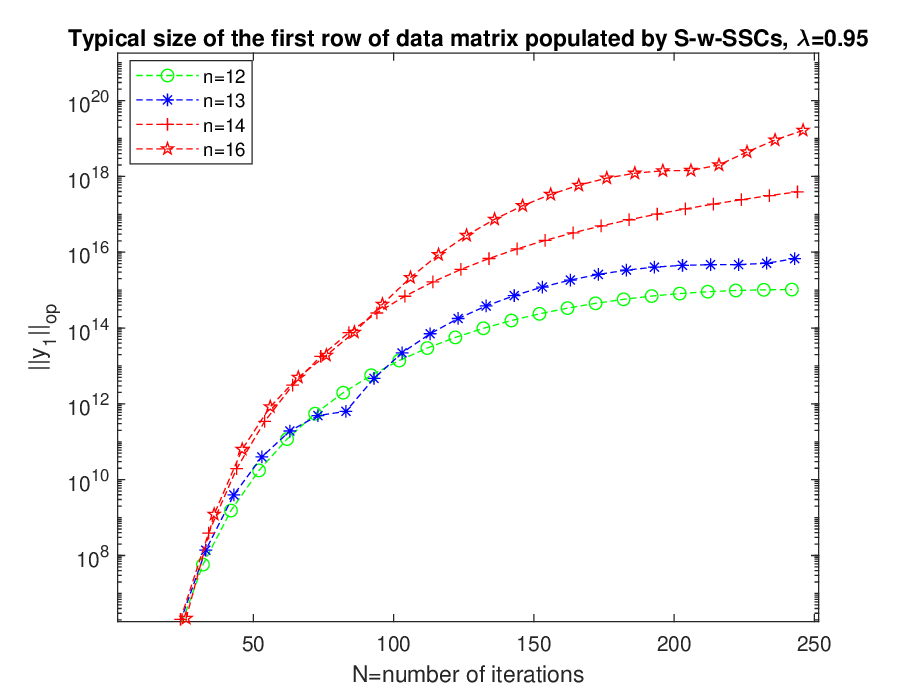}     
\caption{$\ell_{2}$ norm of the data matrixs' first row populated by $n-$ dimensional $(n=12,13,14,16)$ S-w-SSCs(with $\lambda=0.95$) suffers from curse of dimensionality}  
\label{fig:curseofdimtypsizerowone}
\end{center}                
\end{figure}
\begin{figure} [!t]
\begin{center}
\includegraphics[width=0.75\textwidth]{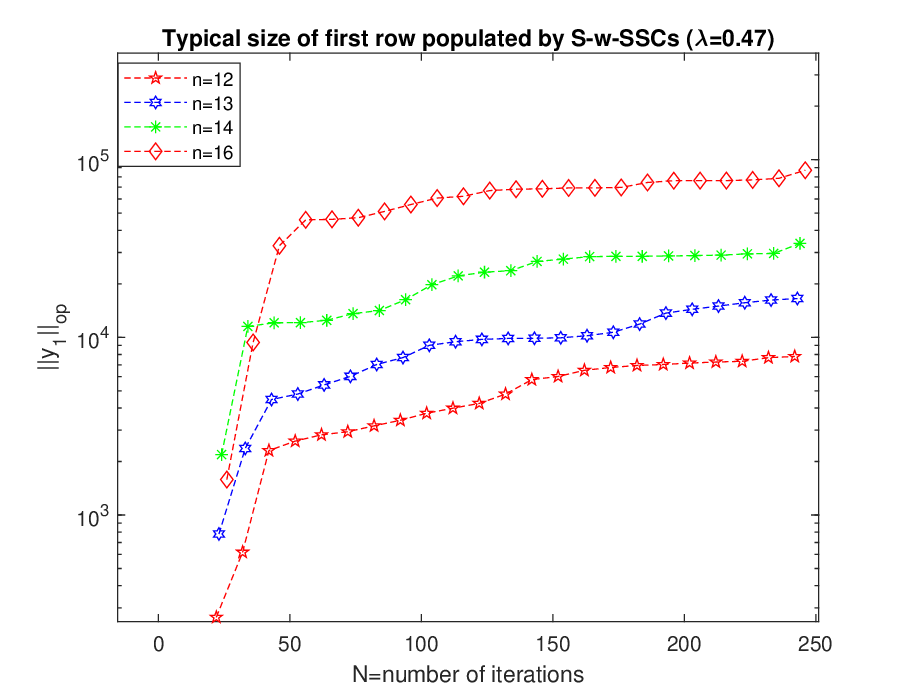}     
\caption{Estimates in Proposition \ref{lm:cov-ul-bd} are optimal: $\ell_{2}$ norm of the data matrixs' first row populated by $n-$ dimensional $(n=12,13,14,16)$ S-w-SSCs(with $\lambda=0.47$) do not suffer from curse of dimensionality } 
\label{fig:nocurserow1}
\end{center}                
\end{figure}
\begin{figure} [!t]
\begin{center}
\includegraphics[width=0.75\textwidth]{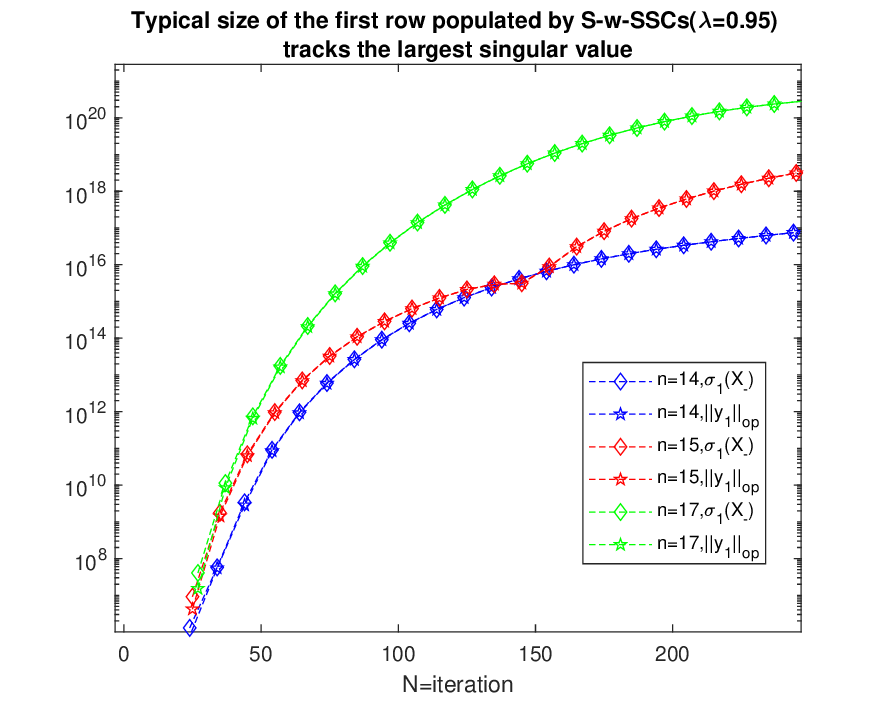}     
\caption{Largest singular value and $\ell_{2}$ norm of the data matrixs' first row populated by $n-$ dimensional $(n=14,15,17)$ S-w-SSCs(with $\lambda=0.95$) are the same} 
\label{fig:rowtrackssig1}
\end{center}                
\end{figure}

\begin{corollary}
    \label{thm:sig_1_exp_n}
    Given $\lambda \in (\frac{1}{2},1)$, almost surely  largest singular of data matrix generated from $n-$ dimensional S-w-SSCs with eigenvalue $\lambda$ suffers from curse of dimnensionality, more precisely:
    \begin{equation}
     \sigma_{1}(X_{-}) \geq \Omega\bigg(\bigg\lfloor \frac{N}{n} \bigg\rfloor^{\frac{1}{2}}  e^{\frac{\alpha_{\lambda}n}{2}} \bigg)
    \end{equation}    
    
\end{corollary}

\section{Eigenvalue Statistics of the Sample covariance Matrix}
\label{sec:eigsmp}
Overwhelming challenge in unraveling the typical order of the eigenvalues associated to the sample covariance matrix(which in turn unravels the fate of OLS) is first establishing the concentration behavior of each individual elements of the sample covariance matrix. 
\begin{gather*}
X_{-}X_{-}^{*}:=
\begin{bmatrix}
     \langle y_{1},y_{1}\rangle &  \ldots & \langle y_{1},y_{k+1} \rangle  &\ldots& \langle y_1,y_{n}\rangle  \\ \vdots & \ldots & \vdots & \ldots & \vdots \\ \langle y_{k+1},y_{1}\rangle &  \ldots & \langle y_{k+1},y_{k+1} \rangle  &\ldots& \langle y_{k+1},y_{n}\rangle \\
     \vdots & \ldots & \vdots & \ldots & \vdots \\
     \langle y_{n},y_{1}\rangle & \ldots & \langle y_{n},y_{k+1} \rangle & \ldots &  \langle y_{n},y_{n}\rangle   
\end{bmatrix}
\end{gather*}
which we will later combine with tools from the perturbation theory to study the localization of all the eigenvalues of the sample covariance matrix.
\subsection{On measure concentration of Sample Covariance matrix} 
\paragraph{Hermitian Case:}
\begin{proposition}
\label{prop:ykyk}
When $A=A^{*}$, stable and w.l.o.g let $\lambda$,$\rho$ be the eigenvalues and $w_t$, $s_t$ be the standard normals corresponding to rows $k$ and $j $, respectively. Then $\langle y_{k},y_{j} \rangle$ is centered (i.e., expectated value is $0$) and there exists positive constants $c_{2,\lambda},c_{4,\lambda},c_{\lambda,\rho},d_{\lambda,\rho},d_{\lambda}$ and $d_{\rho}$ such that: 
\begin{align}
\nonumber &\mathbb{E}(\langle y_{k} , y_{j} \rangle^2)= c_{\lambda,\rho}(N-1)+c_{\lambda,\rho}\bigg[d_{\lambda,\rho}\big(1-o_{(\lambda,\rho);N}(1)\big)-d_{\lambda}\big(1-o_{\lambda;N}(1)\big)-d_{\rho}\big(1-o_{\rho;N}(1)\big) \bigg] \\ \nonumber & \mathbb{E}(\langle y_{k},y_{k}\rangle^{2})=3c_{2,\lambda}^{2}\bigg[(N-1)+c_{4,\lambda}\big(1-o_{\lambda;N}^{2}(1)\big)-2c_{2,\lambda}\big(1-o_{\lambda;N}(1)\big) \bigg].    
\end{align}
where given $|\lambda|<1$, $o_{\lambda;N}(1)$ is used to denote a term vanishing to $0$ as $N$ approaches infinity. 
\end{proposition}
\begin{proof}
Simply notice that:
\begin{align}
    \nonumber & \langle y_{k} , y_{j} \rangle= \sum_{i=1}^{N-1} \bigg(\sum_{t=1}^{i} \lambda^{i-t}w_{t-1}\bigg)\bigg(\sum_{t=1}^{i} \rho^{i-t}s_{t-1}\bigg), \hspace{25pt} \textit{Furthermore}, \hspace{25pt} \mathbb{E} \bigg(\big\langle y_{k} , y_{j} \big \rangle^{2}\bigg) \\ \nonumber &   = \sum_{i=1}^{N-1} \bigg(\sum_{t=1}^{i} \lambda^{2(i-t)}\bigg) \bigg(\sum_{t=1}^{i} \rho^{2(i-t)}\bigg)     =\frac{\sum_{i=1}^{N-1}\big(1-\lambda^{2i}\big)\big(1-\rho^{2i}\big)}{(1-\lambda^{2})(1-\rho^{2})}= \frac{N -\sum_{i=1}^{N-1} (\lambda^{2i}+ \rho^{2i}) +(\lambda \rho)^{2i} }{(1-\lambda^{2})(1-\rho^{2})}.
\end{align}
Diagonal elements are centered at:
\begin{align}
    \mathbb{E}\langle y_{k},y_{k}\rangle=\sum_{i=1}^{N-1}\mathbb{E} \Bigg[\bigg(\sum_{t=1}^{i} \lambda^{i-t}w_{t-1}\bigg)^{2}\Bigg]=\sum_{i=1}^{N-1}\sum_{t=1}^{i}\lambda^{2(i-t)}=c_{2,\lambda}(N-1)-c_{2,\lambda}^{2}\big(1-o_{\lambda;N}(1) \big)
\end{align}
\end{proof}
Now simply recall that: 
\begin{align}
    \mathbb{E}(\langle y_{k},y_{k}\rangle^{2})=\frac{3}{(1-\lambda^2)^{2}}\sum_{i=1}^{N-1} (1-\lambda^{2i})^{2}=\frac{3(N-1)}{(1-\lambda^{2})^{2}}+\frac{3(1-\lambda^{4N})}{(1-\lambda^{2})^{2}(1-\lambda^{4})}-\frac{6(1-\lambda^{2N})}{(1-\lambda^{2})^{2}(1-\lambda^{2})}
\end{align}
Therefore variance of $\langle y_{k} , y_{k}\rangle$
\begin{equation}
3(N-1) c_{2,\lambda}^{2}+3c_{2,\lambda}^{2}c_{4,\lambda}\big(1-o_{\lambda;N}^{2}(1)\big)-6c_{2,\lambda}^{3}\big(1-o_{\lambda;N}(1)\big)-c_{2,\lambda}^{2}(N-1)^{2}-c_{2,\lambda}^{4}\big(1-o_{\lambda;N}(1)\big)^{2}+2c_{2,\lambda}^{3}(N-1)\big(1-o_{\lambda;N}(1)\big)    
\end{equation}
\begin{remark}
\label{rmk:ykyksign}
This implies that $\langle y_{k}, y_{k} \rangle$ is centered at $ c_{2,\lambda}(N-1)-c_{2,\lambda}^{2}\big(1-o_{\lambda;N}(1)\big)$ with fluctuation of size: 
\begin{align}
 \nonumber & \bigg(3(N-1) c_{2,\lambda}^{2}+3c_{2,\lambda}^{2}c_{4,\lambda}\big(1-o_{\lambda;N}^{2}(1)\big)-6c_{2,\lambda}^{3}\big(1-o_{\lambda;N}(1)\big)\\ \nonumber &-c_{2,\lambda}^{2}(N-1)^{2}-c_{2,\lambda}^{4}\big(1-o_{\lambda;N}(1)\big)^{2}+2c_{2,\lambda}^{3}(N-1)\big(1-o_{\lambda;N}(1)\big)\bigg)^{\frac{1}{2}}   
\end{align} 
 On the other hand, now ignoring constants related to eigenvalues $\lambda, \rho$ e.t.c)off-diagonal entries of the sample covariance matrix $(X_{-}X_{-}^{*})$, $\langle y_{k} , y_{j}\rangle \in [-\sqrt{N}, \sqrt{N}]$ w.p 1. and $\langle y_k,y_{k} \rangle \in [N-\sqrt{N}, N+\sqrt{N}]$.   
\end{remark}
\begin{theorem}
    For Hermitian stable case, ignoring all the other parameters besides $n$ and $N$, we have that for all $j \in [n]$:
    \begin{equation}
        N-n\sqrt{N} \leq \lambda_{j}\big( X_{-}X_{-}^{*}\big) \leq N+n\sqrt{N}
    \end{equation}
\end{theorem}
\begin{proof}
    According to Gershgorins' theorem all the eigenvalues of $X_{-}X_{-}^{*}$ lies inside disc centered at $\langle y_{j},y_{j} \rangle$ with radius $\sum_{k \neq j}^{n} |\langle y_{j} , y_{k} \rangle| $ for some $j \in [n]$. Result follows from proposition \ref{prop:ykyk} and preceding remark. 
\end{proof}
\newline
\paragraph{Non-Hermitian case:}
Now we will study concentration for an extreme non-Hermitian case S-w-SSCs: which allows for many intricate spatial and temporal correlations between elements of the data matrix. Evident from the previous section, giving typical size to the elements of the sample covariance matrix will require painstaking attention to time evolution of the dynamics in every row. Since the rows have a causal structure: i.e., $y_{k}$ is itself a function of $[y_{k+1},\ldots,y_{n}]$ so we will start from $y_{n}$. In this working draft, we managed to extract concentration behavior for the last two rows and it becomes evident how off diagonal terms are quiet large which will lead to transient behavior of OLS, but recursive framework of time evolution of one row in terms of rows beneath it will be extremely helpful in getting concentration for all the entries. Writing the rows of the data in compact form when dynamics correspond to $S-w-SSCs$: 
\begin{align}
 \nonumber & \langle y_{j},\hat{e}_{i} \rangle= \sum_{t=1}^{i} \sum_{m=0}^{(i-t) \land (n-j) } \binom{i-t}{m} \lambda^{i-t-m} \big\langle w_{t-1} ,e_{m+j} \big \rangle \\ & \label{eq:recurrow} \textit{Furthermore,} \hspace{10pt}  \langle y_{n-1},\hat{e}_{i+1} \rangle=\lambda \langle y_{n-1},\hat{e}_i \rangle+\langle w_{i},e_{n-1}\rangle+\langle y_{n},\hat{e}_{i}\rangle \\ \nonumber & \textit{Let,} \hspace{20pt} s_{i}:=\sum_{t=1}^{i} \lambda^{i-t}w_{t-1} 
\end{align}
In order to understand the cross correlation between the rows of the data matrix, we leverage upon  \eqref{eq:recurrow}
\begin{align}
    \nonumber  \langle y_{n-1},y_{n}\rangle &= \sum_{i=1}^{N-1}\bigg(\sum_{t=1}^{i}\lambda^{i-t}\langle w_{t-1},e_{n-1}\rangle + \sum_{t=1}^{i-1} \lambda^{i-t-1} \langle w_{t-1},e_{n}\rangle\bigg)\bigg( \sum_{t=1}^{i} \lambda^{i-t}\langle w_{t-1},e_{n}\rangle  \bigg)\\ \nonumber & =\sum_{i=1}^{N-1}\big(\langle s_{i}, e_{n-1} \rangle+\lambda^{-1}\langle s_{i-1}, e_{n}\rangle\big)\langle s_{i}, e_{n}\rangle \\ \nonumber & = \sum_{i=1}^{N-1}\big(\langle s_{i}, e_{n-1} \rangle+\lambda^{-1}\langle s_{i-1}, e_{n}\rangle\big)\big(\lambda \langle s_{i-1}, e_{n}\rangle + \langle w_{i-1}, e_{n}\rangle \big)\\ \nonumber & = \sum_{i=1}^{N-1}\Bigg[ \sum_{t=1}^{i-1}\lambda^{2(i-t)-1}\langle w_{t-1},e_{n}\rangle^{2}+\langle s_{i},e_{n-1} \rangle \langle s_{i}, e_{n}\rangle   + \langle w_{i-1}, e_{n}\rangle \sum_{t=1}^{i}\lambda^{i-t}\langle w_{t-1},e_{n-1}\rangle  \Bigg]
\end{align}
\begin{align}
    \label{eq:eynyn-1} & \mathbb{E}\langle y_{n-1},y_{n}\rangle= \lambda^{-2} \sum_{i=1}^{N-1}\sum_{t=1}^{i-1}\lambda^{4(i-t)}=\frac{1}{\lambda^{-2}(1-\lambda^{4})} \sum_{i=1}^{N-1}\big(1-\lambda^{4i}\big)= \lambda^{-2}c_{4,\lambda}\bigg[(N-1) -c_{4,\lambda}\big(1-o_{\lambda;N}^{2}(1) \big)\bigg] \\ \nonumber & \big(\mathbb{E}\langle y_{n-1},y_{n}\rangle \big)^2=\lambda^{-4}c_{4,\lambda}^{2}\bigg[(N-1) -c_{4,\lambda}\big(1-o_{\lambda;N}^{2}(1) \big)\bigg]^{2}.
\end{align}
Also notice that for $i'>i$, $s_{i'}=\lambda^{(i'-i)}s_{i}+\sum_{t=i+1}^{i'}\lambda^{i'-t}w_{t-1}$, because: 
\begin{equation}
    s_{i'}=\sum_{t=1}^{i'} \lambda^{i'-t}w_{t-1}=\sum_{t=1}^{i} \lambda^{i'-t}w_{t-1} +\sum_{t=i+1}^{i'}\lambda^{i'-t}w_{t-1}=\lambda^{(i'-i)}\sum_{t=1}^{i} \lambda^{i-t}w_{t-1} +\sum_{t=i+1}^{i'}\lambda^{i'-t}w_{t-1} 
\end{equation}
Now consider:
\begin{align}
   \nonumber & \mathbb{E}\bigg(\sum_{i=1}^{N-1}\langle s_{i}, e_{n-1}\rangle \langle s_{i}, e_{n} \rangle \bigg)^{2}=\sum_{i=1}^{N-1}\mathbb{E}\big[\langle s_{i}, e_{n-1}\rangle^{2}\langle s_{i}, e_{n}\rangle^{2}+2\sum_{i'>i}\langle s_{i}, e_{n-1}\rangle \langle s_{i}, e_{n}\rangle \langle s_{i'}, e_{n-1}\rangle \langle s_{i'}, e_{n}\rangle\big], 
\end{align}
where:
\begin{align}
    \nonumber &\sum_{i'>i} \big\langle \lambda^{(i'-i)}s_{i}+\sum_{t=i+1}^{i'}\lambda^{i'-t}w_{t-1}, e_{n-1} \big\rangle \big\langle \lambda^{(i'-i)}s_{i}+\sum_{t=i+1}^{i'}\lambda^{i'-t}w_{t-1}, e_{n} \big\rangle\\ \nonumber &=\sum_{i'>i}\bigg( \big\langle\lambda^{(i'-i)}s_{i},e_{n-1}\big\rangle+\big\langle \sum_{t=i+1}^{i'}\lambda^{i'-t}w_{t-1},e_{n-1}\big\rangle \bigg)\bigg( \big\langle\lambda^{(i'-i)}s_{i},e_{n}\big\rangle+\big\langle \sum_{t=i+1}^{i'}\lambda^{i'-t}w_{t-1},e_{n}\big\rangle \bigg).
\end{align}
Therefore,
\begin{align}
 \nonumber 2\sum_{i=1}^{N-1} \mathbb{E} \langle s_{i}, e_{n-1}\rangle &\langle s_{i}, e_{n}\rangle\sum_{i'>i}\bigg( \langle\lambda^{(i'-i)}s_{i},e_{n-1}\rangle\langle\lambda^{(i'-i)}s_{i},e_{n}\rangle+\langle\lambda^{(i'-i)}s_{i},e_{n-1}\rangle\langle \sum_{t=i+1}^{i'}\lambda^{i'-t}w_{t-1},e_{n}\rangle \\ \nonumber &  +\langle \sum_{t=i+1}^{i'}\lambda^{i'-t}w_{t-1},e_{n-1}\rangle\langle \lambda^{(i'-i)}s_{i},e_{n}\rangle+\langle \sum_{t=i+1}^{i'}\lambda^{i'-t}w_{t-1},e_{n-1}\rangle \langle \sum_{t=i+1}^{i'}\lambda^{i'-t}w_{t-1},e_{n}\rangle\bigg)   
\end{align}
Simplifies to,
\begin{align}
\nonumber & 2\sum_{i=1}^{N-1}\sum_{i'>i}^{N-1}\lambda^{2(i'-i)} \mathbb{E} \langle s_{i},e_{n-1} \rangle^{2} \mathbb{E} \langle s_{i},e_{n} \rangle^{2}=2\sum_{i=1}^{N-1}\sum_{i'>i}^{N-1}\lambda^{2(i'-i)} \bigg(\sum_{t=1}^{i}\lambda^{2(i-t)}\bigg)^{2}\\ \nonumber &=2c_{2,\lambda}^{2}\sum_{i=1}^{N-1}\big(1-\lambda^{2i}\big)^{2}\sum_{i'>i}^{N-1}\lambda^{2(i'-i)}=2c_{2,\lambda}^{2}\sum_{i=1}^{N-1}\big(1-\lambda^{2i}\big)^{2}\sum_{s=1}^{N-1-i}\lambda^{2s}=2c_{2,\lambda}^{3}\sum_{i=1}^{N-1}\big(1-\lambda^{2i}\big)^{2}\big(1-\lambda^{2(N-i)}\big) \\ \nonumber &=2c_{2,\lambda}^{3}\bigg[ (N-1)+c_{4,\lambda}\big(1-o_{\lambda;N}^{2}(1)\big)-2c_{2,\lambda}\big(1-o_{\lambda;N}(1)\big) \bigg] -2c_{2,\lambda}^{3}\sum_{i=1}^{N-1}\big(1-\lambda^{2i}\big)^2\lambda^{2(N-i)},
\end{align}
where we can further simplify
\begin{align}
    \nonumber & 2c_{2,\lambda}^{3}\sum_{i=1}^{N-1}\big(1-\lambda^{2i}\big)^2\lambda^{2(N-i)}=2c_{2,\lambda}^{3}\sum_{i=1}^{N-1} \big( \lambda^{2(N-i)}+ \lambda^{2(N+i)} -2\lambda^{2N} \big)=-4c_{2,\lambda}^{3} \lambda^{2N}\big(N-1\big) \\ \nonumber & +2c_{2,\lambda}^{3}\big( 1+\lambda^{2N}\big) \sum_{i=1}^{N-1} \lambda^{2i}. 
\end{align}
Furthermore, consider the following term:

\begin{align}
 \nonumber & \sum_{i=1}^{N-1} \sum_{t=1}^{i-1}\lambda^{2(i-t)-1}\langle w_{t-1},e_{n}\rangle^{2}= \lambda^{-2}\sum_{t=1}^{N-2} \bigg( \sum_{l=1}^{N-t} \lambda^{2l}  \bigg)\langle w_{t-1}, e_{n} \rangle^{2},    
\end{align}
and 
\begin{align}
\nonumber & \mathbb{E}\Bigg[\lambda^{-2}\sum_{t=1}^{N-2} \bigg( \sum_{l=1}^{N-t} \lambda^{2l}  \bigg)\langle w_{t-1}, e_{n} \rangle^{2}\Bigg]^{2}=3\lambda^{-4}\sum_{t=1}^{N-2} \bigg( \sum_{l=1}^{N-t} \lambda^{2l}  \bigg)^{2}+ 2\lambda^{-4}\sum_{t=1}^{N-2}\sum_{t'>t}^{N-2} \sum_{l=1}^{N-t}   \sum_{l'=1}^{N-t'} \lambda^{2(l'+l)} \\ \nonumber & =3\lambda^{-4}c_{2,\lambda}^{4}\bigg[(N-2)+c_{4,\lambda}\lambda^{4}\big(\lambda^4-o_{\lambda;N}^{2}(1)\big)\bigg]+ 2\lambda^{-4}\sum_{t=1}^{N-2}\sum_{t'>t}^{N-2} \sum_{l=1}^{N-t}   \sum_{l'=1}^{N-t'} \lambda^{2(l'+l)},
\end{align}
where
\begin{align}
    \nonumber & 2\lambda^{-4}\sum_{t=1}^{N-2}\sum_{t'>t}^{N-2} \sum_{l=1}^{N-t}   \sum_{l'=1}^{N-t'} \lambda^{2(l'+l)}= 2c_{2,\lambda}\lambda^{-4}\sum_{t=1}^{N-2}\sum_{t'>t}^{N-2} \sum_{l=1}^{N-t} \lambda^{2l}(1+\lambda^{2(N-t'+1)}) \\ \nonumber &  2c_{2,\lambda}\lambda^{-4}\sum_{t=1}^{N-2}\sum_{t'>t}^{N-2} \big(\sum_{l=1}^{N-t} \lambda^{2l}+ \lambda^{2(N-t'+1)}\sum_{l=1}^{N-t} \lambda^{2l}\big)=2c_{2,\lambda}^{2}\lambda^{-4}\sum_{t=1}^{N-2}\sum_{t'>t}^{N-2} \big([1+\lambda^{2(N-t+1)}]+ \lambda^{2(N-t'+1)}[1+\lambda^{2(N-t+1)}]\big)\\ \nonumber & = c_{2,\lambda}^{2} \lambda^{-4} \Bigg[ \bigg( (N-2)+ \sum_{t=1}^{N-2} \lambda^{2(N-t+1)} \bigg)^2-\sum_{t=1}^{N-2} \big(1+\lambda^{2(N-t+1)}\big)^2  \Bigg]= c_{2,\lambda}^{2} \lambda^{-4}  \bigg( N-2+ c_{2,\lambda} \lambda^{6} \big(1+\lambda^{2(N-2)}\big) \bigg)^2 \\ \nonumber &  - c_{2,\lambda}^{2} \lambda^{-4} \bigg[ (N-2) +c_{4,\lambda}\lambda^{4(3)}\big(1+\lambda^{4(N-2)}\big)+2c_{2,\lambda} \lambda^{2(3)} \big( 1+\lambda^{2(N-2)}\big)  \bigg]\\ \nonumber &=c_{2,\lambda}^{2} \lambda^{-4} \Bigg[ (N-2)^{2} +c_{2,\lambda}^{2}\lambda^{2(6)}\big(1+\lambda^{2(N-2)}\big)^2-(N-2)-c_{4,\lambda}\lambda^{4(3)}\big(1+\lambda^{4(N-2)}\big)-2c_{2,\lambda} \lambda^{2(3)} \big( 1+\lambda^{2(N-2)}\big)  \Bigg] 
\end{align}
Now we are only left with the task of bounding 
\begin{align}
    \nonumber & \mathbb{E}\bigg(\sum_{i=1}^{N-1}\sum_{t=1}^{i}\lambda^{i-t} \langle w_{i-1}, e_{n}\rangle \langle w_{t-1},e_{n-1}\rangle \bigg)^{2}=\sum_{i=1}^{N-1}\mathbb{E}\Bigg[ \big(\sum_{t=1}^{i}\lambda^{i-t} \langle w_{i-1}, e_{n}\rangle \langle w_{t-1},e_{n-1}\rangle \big)^{2} \\ \nonumber & +2\sum_{i'>i}\big(\langle s_{i},e_{n-1} \rangle \langle w_{i-1}, e_{n}\rangle \big)\big(\langle \lambda^{(i'-i)}s_{i}+\sum_{t=i+1}^{i'}\lambda^{i'-t}w_{t-1}, e_{n-1}\rangle \langle w_{i'-1}, e_{n}\rangle\big) \Bigg]\\ \nonumber &=\sum_{i=1}^{N-1}\mathbb{E}\big(\langle \sum_{t=1}^{i}\lambda^{i-t} w_{t-1},e_{n-1}\rangle \langle w_{i-1}, e_{n}\rangle  \big)^{2}=\sum_{i=1}^{N-1}\mathbb{E}\big(\langle s_{i},e_{n-1}\rangle \langle w_{i-1}, e_{n}\rangle \big)^{2}\\ \nonumber &=\sum_{i=1}^{N-1}\mathbb{E}(\langle s_{i},e_{n-1}\rangle^2)= \sum_{i=1}^{N-1}\sum_{t=1}^{i} \lambda^{2(i-t)}=c_{2,\lambda}(N-1)-c_{2,\lambda}^{2}\big(1-o_{\lambda;N}(1) \big)  
\end{align}
\begin{theorem}
Standard deviation of $\langle y_{n-1},y_n \rangle$, $SD\big(\langle y_{n-1},y_n \rangle\big)= O(N-2)$ and there exists positive constants independent of $N$, $\kappa_{\lambda,1}$ and $\kappa_{\lambda,2}$ such that  
    \begin{equation}
        \langle y_{n-1},y_{n}\rangle \in \big[ \kappa_{\lambda,1}, 2N-\kappa_{\lambda,2} \big]
    \end{equation}
\end{theorem}
\begin{proof}
Proof follows the general computational procedure, we know from \eqref{eq:eynyn-1} analytic expression of $\mathbb{E}\langle y_{n},y_{n-1}\rangle$
\begin{align}
        \nonumber & \mathbb{E}\big| \langle y_{n-1},y_{n}\rangle - \mathbb{E}\langle y_{n-1},y_{n}\rangle \big|^{2}=2c_{2,\lambda}^{3}\bigg[ (N-1)+c_{4,\lambda}\big(1-o_{\lambda;N}^{2}(1)\big)-2c_{2,\lambda}\big(1-o_{\lambda;N}(1)\big) \bigg] \\ \nonumber & + c_{2,\lambda}(N-1)-c_{2,\lambda}^{2}\big(1-o_{\lambda;N}(1) \big) -\lambda^{-4}c_{4,\lambda}^{2}\bigg[(N-1) -c_{4,\lambda}\big(1-o_{\lambda;N}^{2}(1) \big)\bigg]^{2} +4c_{2,\lambda}^{3} \lambda^{2N}\big(N-1\big) \\ \nonumber & -2c_{2,\lambda}^{4}\big( 1+\lambda^{2N}\big)^2 +3\lambda^{-4}c_{2,\lambda}^{4}\bigg[(N-2)+c_{4,\lambda}\lambda^{4}\big(\lambda^4-o_{\lambda;N}^{2}(1)\big)\bigg] \\ \nonumber & + c_{2,\lambda}^{2} \lambda^{-4} \Bigg[ (N-2)^{2} +c_{2,\lambda}^{2}\lambda^{2(6)}\big(1+\lambda^{2(N-2)}\big)^2-(N-2)-c_{4,\lambda}\lambda^{4(3)}\big(1+\lambda^{4(N-2)}\big)-2c_{2,\lambda} \lambda^{2(3)} \big( 1+\lambda^{2(N-2)}\big)  \Bigg].  
    \end{align}
    From \eqref{eq:eynyn-1} we know that $\mathbb{E}\langle y_n, y_{n-1}\rangle=O(N-1)$ and result follows.
\end{proof}
\begin{remark}
This is where things start getting tricky, compared to at most typical size of $O(\sqrt{N})$ in Hermitian case for off-diagonal terms from remark \ref{rmk:ykyksign},rows with typical size of $N \pm \sqrt{N}$ and polynomial in $N$ can have a typical size of correlation $\big(N-1\big) \pm O\big(N-2\big)$. As the higher rows will have higher typical size and off-diagonal terms can potentially be $O(N^{n-1})$, which we will study in future work.      
\end{remark}
\paragraph{Typical size of the rows}
As we will shortly see afterwards, in order to get typical order of the bulk eigenvalues, one will need typical order of all the rows, which brings us to:
\begin{proposition}
    Sequnce of random variables $\big\langle y_{j}, \hat{e}_{n-(j-1)} \big\rangle_{j \in [n]}$ are centered at $0$ and:
    \begin{equation}
        O\bigg( \sum_{k=j}^{n-(j-1)} \sum_{m=0}^{n-k} \binom{n-k}{m}^{2} \lambda^{2(n-k-m)} \bigg)^{\frac{1}{2}}
    \end{equation}
    
\end{proposition}
\begin{equation}
 \langle y_{j}, \hat{e}_{i} \rangle= \sum_{t=1}^{i} \sum_{m=0}^{(i-t) \land (n-j) } \binom{i-t}{m} \lambda^{i-t-m} \big\langle w_{t-1} ,e_{m+j} \big \rangle,   
\end{equation}
 
 \begin{equation}
 \langle y_{2},\hat{e}_{n-1}\rangle=\sum_{t=1}^{n-1}\sum_{m=0}^{n-1-t} \binom{n-1-t}{m} \lambda^{n-1-t-m} \langle w_{t-1}, e_{m+2}\rangle,    
 \end{equation}
is centered at 0, weighted sum of independent $\frac{n(n-1)}{2}$ standard normals with variance:
\begin{align}
    4^n \lambda^{2n}L_{\lambda,n}(2) \leq \mathbb{E}\big(\langle y_{2},\hat{e}_{n-1}\rangle^2\big)=\sum_{k=2}^{n}\sum_{m=0}^{n-k} \binom{n-k}{m}^{2} \lambda^{2(n-k-m)}\leq 4^n \lambda^{2n}U_{\lambda,n}(2),
\end{align}
and repeating the recursion we have for $k \in [0, \ldots, n-1]$,
\begin{equation}
\label{eq:intrlacerow}
4^{n}\lambda^{2n}L_{\lambda,n}(k+1)\leq \mathbb{E}(\langle y_{k+1},\hat{e}_{n-k}\rangle^2) \leq  4^{n}\lambda^{2n}U_{\lambda,n}(k+1)
\end{equation}
\subsection{Interlacing the eigenvalues of Sample Covariance matrix }
We will use Courant-Fischer for the distribution of eigenvalues corresponding to the sample covariance matrix, $\Sigma_n:=X_{-}X_{-}^{*}$ 
\begin{equation}
    \lambda_{1}(\Sigma_n) \geq \lambda_{2}(\Sigma_n) \geq \ldots \lambda_{n}(\Sigma_n)>0.
\end{equation}
\begin{gather*}
\Sigma_n:=
\begin{bmatrix}
     \langle y_{1},y_{1}\rangle &  \ldots & \langle y_{1},y_{k+1} \rangle  &\ldots& \langle y_1,y_{n}\rangle  \\ \vdots & \ldots & \vdots & \ldots & \vdots \\ \langle y_{k+1},y_{1}\rangle &  \ldots & \langle y_{k+1},y_{k+1} \rangle  &\ldots& \langle y_{k+1},y_{n}\rangle \\
     \vdots & \ldots & \vdots & \ldots & \vdots \\
     \langle y_{n},y_{1}\rangle & \ldots & \langle y_{n},y_{k+1} \rangle & \ldots &  \langle y_{n},y_{n}\rangle   
\end{bmatrix}
\end{gather*}
$\Sigma_{n-k}$ be the modified sample covariance matrix with first $k-$ columns and rows removed
\begin{gather*}
\Sigma_{n-k}:=
\begin{bmatrix}
     \langle y_{k+1},y_{k+1}\rangle &  \ldots & \langle y_{k+1},y_{n}\rangle  \\ \vdots & \ldots & \vdots \\
     \langle y_{n},y_{k+1}\rangle & \ldots &  \langle y_{n},y_{n}\rangle   
\end{bmatrix}
\end{gather*}
Now we have a task of comparing the eigenvalues of both matrices, letting $V_{i}$ denote $i-$ th dimensional subspace:
\begin{equation}
    \lambda_{i}(\Sigma_n):=\sup_{V_{i}\subset \mathbb{R}^{n}} \hspace{5pt} \inf_{y \in S^{n-1} \cap V_{i}}\big\| \Sigma_n y \big\|, \hspace{5pt} \textit{and} \hspace{5pt} \lambda_{i}\big(\Sigma_{n-k}\big):=\sup_{V_{i}\subset \mathbb{R}^{n-k}} \hspace{5pt} \inf_{z \in S^{n-k-1} \cap V_{i}}\big\| \Sigma_{n-k} z \big\|
\end{equation}
Not hard to realize that for $i=[1, \ldots, n-k]$
\begin{equation}
    \label{eq:ubdbulkintlace}
    \lambda_{i}(\Sigma_n) \geq \lambda_{i}(\Sigma_{n-k}) \geq \lambda_{k+i}(\Sigma_n). 
\end{equation}
Just to be hands on, fix $n=3,k=1$ and vary $i$ from $1$ to $2$ and notice:
\begin{align}
    \nonumber & \lambda_{1}(\Sigma_{3}) \geq \lambda_{1}(\Sigma_{2}) \geq \lambda_{2}(\Sigma_{3}) \geq \lambda_{2}(\Sigma_{2}) \geq \lambda_{3}(\Sigma_3), \hspace{2pt} \textit{Now for:} \hspace{2pt} n=4,k=1,i \in [3] \\ & \label{eq:pyaramid} \lambda_{1}(\Sigma_{4}) \geq \lambda_{1}(\Sigma_{3}) \geq \lambda_{2}(\Sigma_{4}) \geq \lambda_{2}(\Sigma_{3}) \geq \lambda_{3}(\Sigma_{4}) \geq \lambda_{3}(\Sigma_{3}) \geq \lambda_{4}(\Sigma_{4}), 
\end{align}
and one trivially notices that if dynamics follow S-w-SSCs and one increase the dimension, \textbf{Smallest eigenvalue:} will always be upper bounded by $\lambda_{1}(\Sigma_{1})=\langle y_n,y_n \rangle=\|y_{n}\|^{2}$, where recall from \eqref{eq:inpdatswsscs} that:
\begin{equation}
    \langle y_{n},\hat{e}_{i} \rangle=\sum_{t=1}^{i}\lambda^{i-t}\langle w_{t-1},e_{n}\rangle
\end{equation}
i.e., typical size of standard scalar ARMA trajectory of length $N$. Furthermore:  
\begin{theorem} [Upper bounding the smallest singular value for all stable dynamics] Least singular value of the data matrix is essential upper bounded by the typical size of the scalar stable, length $N$ ARMA trajectory. Ignoring constant that can potentially only depend on eigenvalue of $A$,
\begin{equation}
    \sigma_{n}^{2}(X_{-}) \in (0, N+\sqrt{N}], \hspace{5pt} \textit{w.p. 1}.
\end{equation}
\end{theorem}
\begin{proof}
From interlacing theorem we have that:
\begin{align}
    \nonumber  \mathbb{P}\big(\sigma_{n}^{2}(X_{-}) \geq \delta \big)=\mathbb{P}\big(&\lambda_{n}(\Sigma_{n}) \geq \delta \big) \leq \mathbb{P}\big(\lambda_1 (\Sigma_{1})\big)= \mathbb{P}\big(\|y_{n}\|^{2} \geq \delta\big) \\ \nonumber & \frac{\sigma_{n}^{2}(X_{-})}{N} \longrightarrow (0,1].
\end{align}
and as discussed in Remark \ref{rmk:ykyksign}, $\langle y_n, y_n \rangle \in [N-\sqrt{N}, N+\sqrt{N}]$ w.p 1 
\end{proof}
Fix a row of the data matrix $X_{-}$, 
in order to get lower bound and better estimates on the least singular value, one needs to quantify its' distance from the span of remaining rows of data matrix. As we can take conjugate transpose and write least singular value of the data matrix as:
\begin{equation}
    \sigma_{n}(X_{-}):= \inf_{a \in S^{n-1}} \|X_{-}^{*}a\|_{2}
\end{equation}
Leveraging upon the Hilbertian structure(orthogonal projections/inner products), given any subspace $V \subset \mathbb{R}^{n}$, we have that:
\begin{equation}
    \|X_{-}^{*}a\|_{2}^{2}= \|P_{V}(X_{-}^{*}a)\|_{2}^{2} + \|P_{V^{\perp}}(X_{-}^{*}a)\|_{2}^{2} \geq \|P_{V^{\perp}}(X_{-}^{*}a)\|_{2}^{2}. 
\end{equation}
Therefore,
\begin{align}
    \big\|X_{-}^{*}a\big\|_{2} \geq |a_{i}|d(y_i,n_i), \hspace{5pt} \textit{for all} \hspace{3pt} i \in [n].
\end{align}
Now if $a \in S^{n-1}$, then there exists an $i \in [n]$ such that $|a_{i}| \geq \frac{1}{\sqrt{n}}$. Subsequently,
\begin{align}
    P\bigg( \inf_{a \in S^{n-1}} \|X_{-}^{*}a\| \leq \delta  \bigg) =P\bigg( \exists \hspace{3pt} a \in S^{n-1}: \hspace{2pt} \|X_{-}^{*}a\| \leq \delta  \bigg) \leq  P\bigg( \exists \hspace{3pt} j \in [n]: \hspace{2pt} \frac{d(y_j,n_j)}{\sqrt{n}} \leq \delta  \bigg)
\end{align}
\begin{enumerate}
    \item when we have independent rows union bound is not terribly wasteful, so: 
    \begin{align}
        \mathbb{P}\big(\sigma_{n}(X_{-}) \leq \delta \big) \leq \sum_{j=1}^{n} \mathbb{P}\bigg(d(y_j,n_j)\leq  \delta \sqrt{n} \bigg)
    \end{align}
    Let $w_{j} \in S^{N-1}$ be orthogonal unit vector to subspace $n_j$(not necessarily unique), if the rows are independent then $w_{j}$ can be chosen independently of $y_{j}$. Therefore:
    \begin{equation}
        \langle w_{j},y_{j}\rangle =\langle w_{j},P_{n_{j}^{\perp}}y_{j}\rangle \leq \big\| P_{n_{j}^{\perp}}y_{j}\big\|=d(y_{j},n_{j}),
    \end{equation}
    where inequality follows from Cauchy-Schwarz. Consequently:
    \begin{equation}
        \mathbb{P}\bigg(d(y_j,n_j)\leq \delta\sqrt{n}\bigg) \leq \mathbb{P}\bigg(\langle w_{j},y_{j}\rangle\leq \delta\sqrt{n}\bigg), 
    \end{equation}
    where controlling the last term is commonly studied under the name of \emph{anti-concentration or small ball probability}.
    \item rows with strong correlation(S-w-SSC-s), then union bound will essentially conclude vacuous estimates so its' better to bound
    \begin{align}
        \mathbb{P}\big(\sigma_{n}(X_{-}) \leq \delta \big) \leq   \mathbb{P}\bigg(\min_{j \in [n]} d(y_j,n_j)\leq \delta \sqrt{n} \bigg)  
    \end{align}
\end{enumerate}
\textbf{Taking care of the bulk:} One strategy for upper bounding the bulk eigenvalues follows naturally from interlacing given in \eqref{eq:ubdbulkintlace}, i.e., for $k \in [n]$
\begin{equation}
    \lambda_{k}(\Sigma_{n}) \leq \lambda_{1}(\Sigma_{n-k+1}),
\end{equation}
when combined with typical size of all the rows from Theorem \eqref{eq:intrlacerow} translates into
\begin{equation}
\sqrt{4^{n}\lambda^{2n}L_{\lambda,n}(k+1) \bigg\lfloor \frac{N}{n-k+1} \bigg\rfloor} \leq  \sqrt{\lambda_{1}(\Sigma_{n-k+1})} \leq \sqrt{\big(N-n+\big) 4^{n}\lambda^{2n}U_{\lambda,n}(k+1)},
\end{equation}
and consequently a lower bound on estimation error:
 \begin{align}
     \nonumber \big\|A-\hat{A} \big\|_{F} \geq \sigma_{n}(EX_{-}^{*}) \sqrt{\sum_{j=1}^{n} \frac{1}{\lambda_{1}^{2}\big(\Sigma_{n-k+1}\big)}}. 
\end{align}
In fact, when we have typical size of all the entries of sample covariance matrix we can probably get even better estimates on the bulk via tools from Quadratic constrained optimization. For example, using min-max characterization of the eigenvalues:
\begin{equation}
    \lambda_{1}(\Sigma_{2})= \max{\big(\sqrt{\langle y_{n-1},y_{n-1}\rangle^{2}+ \langle y_{n-1},y_{n}\rangle^{2}},  \sqrt{\langle y_{n},y_{n}\rangle^{2}+ \langle y_{n-1},y_{n}\rangle^{2}}  \big)}
\end{equation}
So in the case of S-w-SSCs, $ \lambda_{1}(\Sigma_{2})= \langle y_{n-1},y_{n-1}\rangle^{2}+ \langle y_{n-1},y_{n}\rangle^{2}$ and
\begin{equation}
\lambda_{2}(\Sigma_{2})= \inf_{a \in S^{1}} \sqrt{(a_{1} \langle y_{n-1},y_{n-1} \rangle+ a_{2}\langle y_{n},y_{n-1}\rangle)^2+ (a_{1} \langle y_{n-1},y_{n}\rangle+ a_{2}\langle y_{n},y_{n}\rangle)^2}= \inf_{a \in S^{1}} \|\Sigma_{2}a\|.    
\end{equation}
which is a quadratic constrained optimization problem.
Again using the preceding argument:
\begin{align}
    \nonumber  \lambda_{1}^{2}(\Sigma_3) = & \big(\|y_{n-2}\|^{2}+ \langle y_{n-2},y_{n-1}\rangle^{2}+\langle y_{n-2}, y_{n}\rangle^{2}\big) \vee \big(\langle y_{n-2},y_{n-1}\rangle^{2}+ \|y_{n-1}\|^{2}+\langle y_{n-1},y_{n}\rangle^{2}\big) \\ \nonumber & \vee  \big(\|y_n\|^{2}+ \langle y_{n-2} y_{n}\rangle^{2}+\langle y_{n-1},y_{n}\rangle^2\big)
\end{align}
For S-w-SSCs case: guess $\lambda_{1}(\Sigma_{3})=\sqrt{\langle y_{n-2},y_{n-2}\rangle^{2}+ \langle y_{n-2},y_{n-1}\rangle^{2}+\langle y_{n-2},y_{n}\rangle^{2}}$, and 
\begin{align}
\nonumber \lambda_{2}^{2}(\Sigma_{3})= \inf_{a \in S^{1}}\bigg[ (a_{1} \langle y_{n-2},y_{n-2} \rangle+ a_{2}\langle y_{n-2},y_{n-1}\rangle)^2+& (a_{1} \langle y_{n-1},y_{n-2}\rangle+ a_{2}\langle y_{n-1},y_{n-1}\rangle)^2 \\ \nonumber & + (a_{1} \langle y_{n},y_{n-2}\rangle+ a_{2}\langle y_{n},y_{n-1}\rangle)^2\bigg],
\end{align}
which we may be able to explicitly compute using optimization principles and is a part of future plan of the authors' research.
\subsection{Soft edge statistics of the martingale transform}
Another evident quantity that appears in estimation error of least squares is edge spectral statistics of the commonly known martingale transform term $EX_{-}^{*}$. For a quick insight into the largest singular value of martingale transform assume $A=A^{*}$ with only one eigenvalue $\lambda$, so
\begin{gather*}
EX_{-}^{*}=\sum_{i=1}^{N-1}\sum_{t=1}^{i}
\begin{bmatrix}
     \langle w_{i},e_{1}\rangle \langle   w_{t-1},e_{1}\rangle& \ldots &  \langle w_{i},e_{1}\rangle \langle w_{t-1},e_{n}\rangle  \\ \vdots & \ddots & \vdots \\
     \langle w_{i},e_{n}\rangle \langle   w_{t-1},e_{1}\rangle & \ldots  &\langle w_{i},e_{n}\rangle \langle  w_{t-1},e_{n}\rangle   
\end{bmatrix}
\end{gather*}
now pick  $(1,0,\ldots,0)=:a \in S^{n-1}$ then , then denote $s_{i}:= \sum_{t=1}^{i}w_{t-1}$ and notice
\begin{equation}
    s_{i+m}=\sum_{t=1}^{i+m}w_{t-1}=\underbrace{\sum_{t=1}^{i}w_{t-1}}_{s_i}+\sum_{t=i+1}^{i+m}w_{t-1}
\end{equation}
\begin{align}
& \nonumber \mathbb{P}\bigg(\big\|EX_{-}^{*}a\big\|\geq \delta \bigg) \leq \delta^{-2} \sum_{j=1}^{n}\mathbb{E}\bigg(\sum_{i=1}^{N-1} \langle w_{i},e_{j}\rangle  \langle \underbrace{\sum_{t=1}^{i}  w_{t-1}}_{=:s_{i}},e_{1}\rangle \bigg)^{2} \\ & \nonumber \delta^{-2}\sum_{j=1}^{n} \sum_{i=1}^{N-1}\mathbb{E} (\langle w_{i},e_{j}\rangle \langle s_{i}, e_1 \rangle)^{2}+2\langle w_{i},e_{j}\rangle \langle s_{i}, e_1 \rangle \sum_{m=1}^{N-1-i} \langle w_{i+m},e_{j}\rangle \langle s_{i+m}, e_1 \rangle \\ & \nonumber \delta^{-2}\sum_{j=1}^{n} \sum_{i=1}^{N-1}\mathbb{E} \big[(\langle w_{i},e_{j}\rangle \langle s_{i}, e_1 \rangle)^{2}+2\langle w_{i},e_{j}\rangle \langle s_{i}, e_1 \rangle \sum_{m=1}^{N-1-i}(\langle w_{i+m},e_{j}\rangle \langle s_{i}, e_1 \rangle+\langle w_{i+m},e_{j}\rangle \langle \sum_{t=i+1}^{i+m}w_{t-1},e_{1} \rangle) \big] \\ & \nonumber= \delta^{-2}\sum_{j=1}^{n} \sum_{i=1}^{N-1}\mathbb{E} \big[(\langle w_{i},e_{j}\rangle \langle s_{i}, e_1 \rangle)^{2}+2\langle w_{i},e_{j}\rangle \langle s_{i}, e_1 \rangle \sum_{m=1}^{N-1-i}\langle w_{i+m},e_{j}\rangle \langle w_{i},e_{1}\rangle+ \langle w_{i+m},e_{j}\rangle \sum_{t'=i+2}^{i+m} \langle w_{t'-1},e_{1} \rangle \big] \\ & \nonumber= \delta^{-2} \sum_{j=1}^{n}\sum_{i=1}^{N-1} \mathbb{E} \big[ \langle w_{i},e_{j}\rangle \langle w_{i},e_{j}\rangle \langle s_{i}, e_1 \rangle \langle s_{i}, e_1 \rangle\big]\\ & \nonumber=\delta^{-2}\sum_{j=1}^{n}\sum_{i=1}^{N-1} \mathbb{E}[(\langle w_{i},e_{j}\rangle)^2]\mathbb{E}[(\langle s_{i},e_{1}\rangle)^{2}]= \delta^{-2}\sum_{j=1}^{n}\sum_{i=1}^{N-1}i=\frac{n(N-1)N}{\delta^2}.
\end{align}
Now one can trivially see that if $a=e_{j} \in S^{n-1}$, regardless of $j \in [n]$, $\|EX_{-}^{*}a\|=\frac{n(N-1)N}{\delta^2}$. Therefore, $\sigma_{1}(EX_{-}^{*})=O\big(\frac{n(N-1)N}{\delta^2}\big)$. Now for Hermitian stable case, with only eigenvalue $\lambda$,
\begin{align}
    & \nonumber \mathbb{E}\bigg(\sum_{i=1}^{N-1} \langle w_{i},e_{j}\rangle  \langle \underbrace{\sum_{t=1}^{i} \lambda^{i-t} w_{t-1}}_{=:s_{i}(\lambda)},e_{1}\rangle \bigg)^{2}=\sum_{j=1}^{n}\sum_{i=1}^{N-1}\mathbb{E}[\langle s_{i}(\lambda),e_{1}\rangle ^{2}]=\sum_{j=1}^{n}\sum_{i=1}^{N-1}\sum_{t=1}^{i} \lambda^{2(i-t)}\\ & \nonumber \textit{Therefore,} \hspace{72pt} \sigma_{1}(EX_{-}^{*})=O\bigg(\frac{n(N-\sum_{i=0}^{N-1}\lambda^{2i})}{1-\lambda^{2}} \bigg).
\end{align}
For the peculiar case of S-w-SSCs with $\lambda \in \big(\frac{1}{2},1\big)$, notice that
\begin{align}
    EX_{-}^{*}=\sum_{i=1}^{N-1}w_{i}x_{i}^{*}=EX_{-}^{*}=\sum_{i=1}^{N-1}w_{i}\bigg(\sum_{t=1}^{i}\sum_{m=0}^{(i-t)\land n}\binom{i-t}{m}N_{n}^{m} \lambda^{i-t-m} w_{t-1}\bigg)^{*}  
\end{align}
Now for $j,k \in [n]$
\begin{align}
     \nonumber [EX_{-}^{*}]_{j,k}&=\sum_{i=1}^{N-1}\sum_{t=1}^{i}\sum_{m=0}^{(i-t)\land n}\binom{i-t}{m}\lambda^{i-t-m}\big\langle w_{i}w_{t-1}^{*}N_{n}^{m^{*}}e_{k} ,e_{j} \big\rangle\\ \nonumber&=\sum_{i=1}^{N-1}\sum_{t=1}^{i}\sum_{m=0}^{(i-t)\land (n-k)}\binom{i-t}{m}\lambda^{i-t-m}\big\langle w_{i}w_{t-1}^{*}e_{k+m} ,e_{j} \big\rangle\\ \nonumber &  = \sum_{i=1}^{N-1}\sum_{t=1}^{i}\sum_{m=0}^{(i-t)\land (n-k)}\binom{i-t}{m}\lambda^{i-t-m}\langle w_{i},e_{j} \rangle \langle w_{t-1},e_{k+m} \rangle
\end{align}
Now in order to find the order of $\sigma_{1}(EX_{-}^{*})$, we need to compute
\begin{align}
    \nonumber& \hspace{18pt}  \max_{k \in [n]} \sum_{j=1}^{n} \mathbb{E}\Bigg( \sum_{i=1}^{N-1}\langle w_{i},e_{j} \rangle \sum_{t=1}^{i}\sum_{m=0}^{(i-t)\land (n-k)}\binom{i-t}{m}\lambda^{i-t-m} \langle w_{t-1},e_{k+m} \rangle\Bigg)^{2}\\  \nonumber&= \hspace{5pt}\max_{k \in [n]} \sum_{j=1}^{n}\sum_{i=1}^{N-1} \mathbb{E}\bigg(\sum_{t=1}^{i}\sum_{m=0}^{(i-t)\land (n-k)}\binom{i-t}{m}\lambda^{i-t-m} \langle w_{t-1},e_{k+m} \rangle\bigg)^2 \\  \label{eq:rowdensesimsig1} &= \hspace{5pt} \sum_{j=1}^{n}\sum_{i=1}^{N-1} \mathbb{E}\bigg(\sum_{t=1}^{i}\sum_{m=0}^{(i-t)\land (n-1)}\binom{i-t}{m}\lambda^{i-t-m} \langle w_{t-1},e_{1+m} \rangle\bigg)^2 = \sum_{j=1}^{n}\mathbb{E}\big\|y_1\big\|^{2}=n\mathbb{E}\|y_{1}\|^{2}, 
\end{align}
where \eqref{eq:rowdensesimsig1} follows along the similar argument in deducing $\sigma_{1}(X_{-})$ of $S-w-SSCs$; first row has the largest typical size. This bring us to a remarkable observation
\begin{theorem}
Largest singular value of martingale term is upper bounded by $\sqrt{n}$ times of largest singular value of the data matrix 
    \begin{align}
          & \nonumber  \sigma_{1}(EX_{-}^{*})\leq \sqrt{n}\sigma_{1}(X_{-}),   
    \end{align}    
\end{theorem}
In fact this was first mentioned in \cite{naeem2023high} by using Courant-Fischer that $\sigma_{1}(EX_{-}^{*})\leq \sigma_{1}(E_{Im(X_{-}^{*})})\sigma_{1}(X_{-})$ where $E_{Im(X_{-}^{*})}$ is essentially $n \times n$ Gaussian ensemble, known to have largest singular value of order $\sqrt{n}$.
\section{Non-vanishing error for stable-spatially inseparable case via tensorization of Talagrands' inequality: Heurestics}
\label{sec:Tal}
Dimension+iteration independent tensorization of Talagrands' inequality would require sufficient independence between the covariates. In fact tensorization for stable ARMA processes had been of much interest to people working in the field of \emph{Functional Inequalities} in early $2000s$. \cite{djellout2004transportation} was able to show dimension+iteration independent tensorization under the assumption of $\|A\|_{2}<1$. \cite{blower2005concentration} concludes dimension+iteration independent tensorization under the assumption of spectral radius being strictly inside unit circle, which turns out to be incorrect in high dimensions.
\begin{definition}[Dimension + Iteration free Tensorization]
Let $(x_1,x_2,\ldots,x_{N}) \in \mathbb{R}^{n^{\otimes N}}$ be a length $N-$ trajectory of an $n-$ dimensional linear Gaussian as in \eqref{eq:LGS}, and $\mu_{n}^{N}:=Law(x_1,x_2, \ldots, x_{N})$ be its' distribution. We say that the process $(x_1,x_2,\ldots,x_{N})$ satisfies dimension+iteration independent Talagrands' inequality, if for all probability measures $\nu$ on $\mathbb{R}^{n^{\otimes N}}$ which are absolutely constant w.r.t $\mu_{n}^{N}$( $\nu << \mu_{n}^{N}$), we have that
\begin{equation}
    W_{1}^{d}\big(\mu_{n}^{N},\nu \big) \leq \sqrt{2CH \big(\nu|| \mu_{n}^{N}\big)},
\end{equation}
where $C$ is independent of $N$ and $n$.
\end{definition}

\textbf{Question:} Now given that dynamics are generated from stable $n-$ dimensional linear Gaussian (i.e., $\rho(A)<1$) can we conclude dimension+iteration independent tensorization of Talagrands' inequality?

The slickest way to conclusion is a realization that $(x_{1},x_{2}, \ldots,x_{N})$ is a function of the Gaussian ensemble $(w_{0},w_{1},\ldots,w_{N-1})$ as $x_{i}=\sum_{t=1}^{i}A^{i-t}w_{t-1}$. Let $F(w_{0},w_{1}, \ldots,w_{N-1}):=\sqrt{\sum_{i=1}^{N} \big\|x_{i}\big\|^{2}}$. As isotropic Gaussians satisfy dimension+iteration independent tensorization (Theorem \ref{thm:dim_ind_tal}), when combined with equivalent notion of Lipschitz concentration( Remark \ref{rm: lipiid}) implies for all $\delta>0$:
\begin{equation}
    \mathbb{P}\bigg(\big|F(w_0,w_1,\ldots,w_{N-1})- \mathbb{E}F(w_0,w_1,\ldots,w_{N-1})\big|> \delta \bigg) \leq 2 \exp{\bigg( -\frac{\delta^{2}}{2\|F\|_{L}^{2}} \bigg)},
\end{equation} 
but the map $(x_{1},x_{2}, \ldots,x_{N}) \mapsto \sqrt{\sum_{i=1}^{N} \big\|x_{i}\big\|^{2}}$ is one-Lipschitz in $\ell_{2}$ additive metric on $\mathbb{R}^{n^{\otimes N}}$ and Remark \ref{rm: lipiid} implies that $\mu_{n}^{N} \in T_{1}\big( \|F\|_{L}^{2}\big)$.
We first need to collect some estimates that will lead to conclusive answer to tensorization of Talagrands' inequality for stable $n-$ dimensional ARMA model.
\begin{proposition}
    \label{prop:frobdataSwSSCs}
    Frobenius norm of the data matrix populated by $n-$ dimensional S-w-SSCs is precisely:
    \begin{align}
        \big\|X_{-}\big\|_{F}^{2}= \sum_{i=1}^{N} \sum_{j=1}^{n} \sum_{s,t=1}^{i} \sum_{m=0}^{(i-t)\land (n-j)} \sum_{m'=0}^{(i-s)\land (n-j)}  \binom{i-t}{m} \binom{i-s}{m'} \lambda^{i-t-m} \overline{\lambda^{i-s-m'}} \langle w_{t-1},e_{j+m}\rangle \overline{\langle w_{s-1},e_{j+m'}\rangle}
    \end{align}
\end{proposition}
\begin{proof}
We again leverage upon the inner product structure to conclude. 
 \begin{align}
    & \nonumber \sum_{i=1}^{N}\big\| x_{i} \big\|^{2}= \sum_{i=1}^{N} \sum_{s,t=1}^{i} \sum_{m=0}^{(i-t)\land n} \sum_{m'=0}^{(i-s)\land n}\binom{i-t}{m} \binom{i-s}{m'} \lambda^{i-t-m} \overline{\lambda^{i-s-m'}} \big\langle N_{n}^{m} w_{t-1}, N_{n}^{m'} w_{s-1} \big \rangle \\ & \nonumber =   \sum_{i=1}^{N} \sum_{s,t=1}^{i} \sum_{m=0}^{(i-t)\land n} \sum_{m'=0}^{(i-s)\land n}\binom{i-t}{m} \binom{i-s}{m'} \lambda^{i-t-m} \overline{\lambda^{i-s-m'}}\sum_{j=1}^{ m \land m'} \langle w_{t-1},e_{j+m}\rangle \overline{\langle w_{s-1},e_{j+m'}\rangle} \\ & \nonumber = \sum_{i=1}^{N} \sum_{s,t=1}^{i} \sum_{m=0}^{(i-t)\land n} \sum_{m'=0}^{(i-s)\land n} \sum_{j=1}^{ m \land m'} \binom{i-t}{m} \binom{i-s}{m'} \lambda^{i-t-m} \overline{\lambda^{i-s-m'}} \langle w_{t-1},e_{j+m}\rangle \overline{\langle w_{s-1},e_{j+m'}\rangle} \\ & \nonumber= \sum_{i=1}^{N} \sum_{j=1}^{n} \sum_{s,t=1}^{i} \sum_{m=0}^{(i-t)\land (n-j)} \sum_{m'=0}^{(i-s)\land (n-j)}  \binom{i-t}{m} \binom{i-s}{m'} \lambda^{i-t-m} \overline{\lambda^{i-s-m'}} \langle w_{t-1},e_{j+m}\rangle \overline{\langle w_{s-1},e_{j+m'}\rangle}
\end{align}   
\end{proof}
\begin{theorem}
\label{thm:Tallindim}
$\|F\|_{L}$ is the smallest positive constant $L_{n,N}$ such that:
    \begin{equation}
        \big\|X_{-}\big\|_{F}=\sqrt{\sum_{i=1}^{N} \big\langle \sum_{t=1}^{i}A^{i-t}w_{t-1}, \sum_{s=1}^{i}A^{i-s}w_{s-1} \big \rangle} \leq L_{n,N} \sqrt{\sum_{i=0}^{N-1}\big\|w_{i}\big\|^{2}}=L_{n,N}\big\|E\big\|_{F},
    \end{equation}
in the limit of large number of iterations we have,    
\begin{itemize}
    \item Hermitian and stable case $L_{n,N}=\Theta(1)$: i.e, independent of dimension+iteration
    \item Stable but spatially inseparable case with $\lambda \in (\frac{1}{2},1)$, then  $L_{n,N}=\Theta(e^{n})$: i.e., linear in underlying dimension of the state space and independent of iterations.
\end{itemize}   

\end{theorem}

\begin{figure} [!t]

\begin{center}
\includegraphics[width=0.75\textwidth]{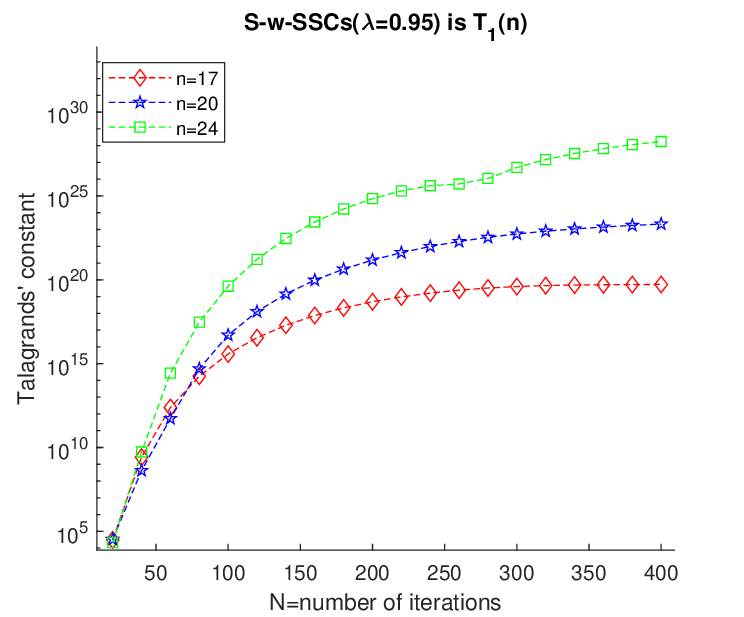}     
\caption{Tensorization is exponential in $n-$ and independent of $N$ for large $N$. Verified for $(n=17,20,24)$ S-w-SSCs(with $\lambda=0.95$) } 
\label{fig:Tal}
\end{center}                
\end{figure} 

\begin{proof} of 
 Hermitian and stable case is obvious after recalling spectral theorem(i.e., all the rows are essentially one-dimensional scalar stable ARMA). Let $[\lambda_{j}]_{j=1}^{n}$ be the eigenvalues of $A$ which are assumed to be strictly inside the unit circle.
\begin{align}
 & \nonumber \sum_{i=1}^{N-1} \sum_{j=1}^{n} \big| \sum_{t=1}^{i}\langle A^{i-t}w_{t-1},e_{j}\rangle \big|^{2}=\sum_{i=1}^{N-1} \sum_{j=1}^{n} \big| \sum_{t=1}^{i}\langle \lambda_{j}^{i-t}w_{t-1},e_{j}\rangle \big|^{2} \leq L_{n,N}^{2} \sum_{i=0}^{N-1} \|w_{i}\|^{2} \\ & \label{eq:talblowherm} \textit{where,} \hspace{5pt} L_{n,N} \leq \min_{\epsilon>0} \max_{j \in [n]} \sqrt{(1+\epsilon^{-1}) \sum_{i=0}^{N-1}\big( (1+\epsilon)|\lambda_{j}|^2\big)^{k}} \leq \frac{1}{1-\rho},     
\end{align}   
where proof of \eqref{eq:talblowherm} is given in Proposition 4.1 of \cite{blower2005concentration}.
However, most difficult but most interesting case is for trajectory generated from $n-$ dimensional $S-w-SSCs$ and $\lambda \in (\frac{1}{2},1)$: and an immediate daunting approach suggested by Prooposition \ref{prop:frobdataSwSSCs} would be a tight estimate on $L_{n,N}$ such that:
\begin{align}
    \nonumber \sum_{i=1}^{N} \sum_{j=1}^{n} \sum_{s,t=1}^{i} \sum_{m=0}^{(i-t)\land (n-j)} \sum_{m'=0}^{(i-s)\land (n-j)}  \binom{i-t}{m} \binom{i-s}{m'} &  \lambda^{i-t-m} \overline{\lambda^{i-s-m'}} \langle w_{t-1},e_{j+m}\rangle \overline{\langle w_{s-1},e_{j+m'}\rangle} \\ & \nonumber \leq L_{n,N}^{2} \sum_{i=0}^{N-1}\big\|w_{i}\big\|^{2}.
\end{align}
Recall from corollary \ref{thm:sig_1_exp_n} that $\sigma_{1}(X_{-})$ has a typical size exponential in $n$.Now recall that $\big\|E\big\|_{F}=\sqrt{\sum_{i=0}^{N-1}\big\|w_{i}\big\|^{2}}$ has a typical size of $\Theta\big(\sqrt{nN}\big)$. Notice that, for S-w-SSCs Talagrands' constant is $\Theta(e^{n})$, thus:
\begin{equation}
    L_{n,N}=\frac{\sqrt{\sum_{i=1}^{n} \sigma_{i}^{2}(X_{-})}}{\sqrt{\sum_{i=1}^{n} \sigma_{i}^{2}(E)}}= \frac{\sqrt{\sum_{i=2}^{n} \sigma_{i}^{2}(X_{-}) + \Theta(\sqrt{N-n+1}e^{n})^{2}}}{\Theta(\sqrt{nN})},    
\end{equation}
and the result follows.
\end{proof}
\begin{corollary}
    In high dimensions, OLS estimation on data corresponding to S-w-SSCs has non-vanishing error in Frobenius norm. 
\end{corollary}
\begin{proof}
 Using Courant-Fischer and the fact that $\sigma_{n}(E)=\sqrt{N}-\sqrt{n-1}$ with high probability(see e.g., \cite{rudelson2014recent}) we conclude:
\begin{align}
     & \nonumber \big\|A-\hat{A} \big\|_{F}= \big\|EX_{-}^{*}(X_{-}X_{-}^{*})^{-1} \big\|_{F} \geq \frac{\big\|EX_{-}^{*} \big\|_{F}}{\sigma_{1}^{2}(X_{-})} \geq \frac{\sigma_{n}(E) \|X_{-}\|_{F}}{\sigma_{1}^{2}(X_{-})}\\ & \nonumber \hspace{50pt} = \frac{(\sqrt{N}-\sqrt{n-1})\sqrt{nN}e^{n}}{(N-n+1)e^{n}}= \frac{(\sqrt{N}-\sqrt{n-1})\sqrt{nN}}{(N-n+1)}>0
\end{align}   
\end{proof}
\begin{figure} [!t]
\begin{center}
\includegraphics[width=0.75\textwidth]{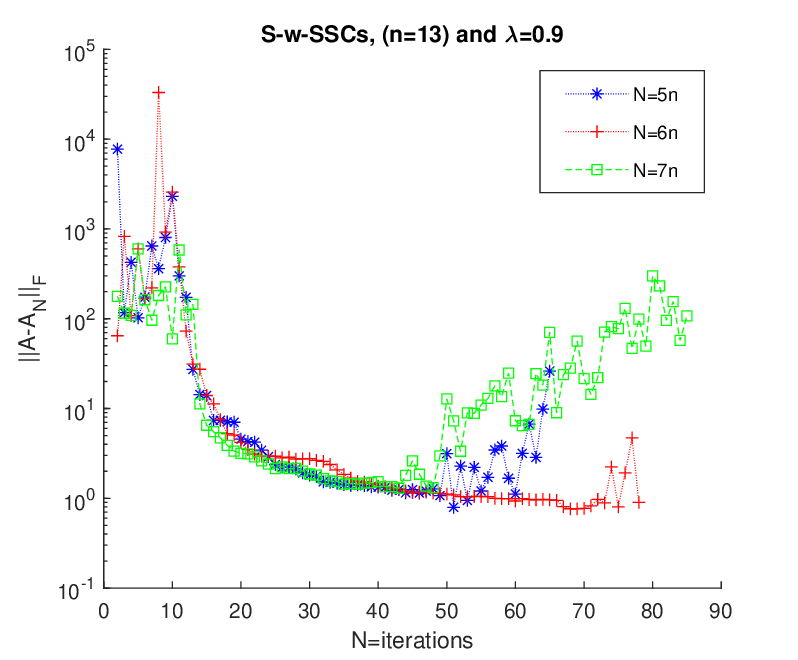}     
\caption{Estimation error worsening with increase in iterations} 
\label{fig:OLS_Sensitive_Iteration}
\end{center}                
\end{figure}
\section{Conclusion}
\label{sec:conc}
We have managed to make significant progress in giving typical order explicit in $N$ and $n$ of various spectral statistics that determine performance of OLS. Although OLS is transient in spatially inseparable case but their does exist a sweet spot on length of the simulated trajectory so that the error is relatively small. In order to find that regime we will need to extend the typical order analysis in section \ref{sec:spst} to higher typical sized rows.

\bibliography{example_paper}
\end{document}